\documentclass[11pt]{article}

\usepackage{amsmath, amssymb,latexsym,graphicx,mfpic}

\def\Z{{\mathbb Z}}
\def\irr{{\rm irr}}
\def\GL{{\rm GL}}
\def\SL{{\rm SL}}

\def\P{{\mathbb P}}
\def\Disc{{\rm Disc}}
\def\Aut{{\rm Aut}}
\def\Vol{{\rm Vol}}
\def\R{{\mathbb R}}
\def\F{{\mathbb F}}
\def\FF{{\mathcal F}}
\def\RR{{\mathcal R}}
\def\Q{{\mathbb Q}}
\def\H{{\mathcal H}}

\def\P{{\mathbb P}}
\def\FF{\mathcal{F}}

\def\C{\mathbb{C}}

\def\Gal{{\rm Gal}}

\def\Disc{{\rm Disc}}

\def\O{{\mathcal O}}
\def\CC{{\mathcal C}}

\def\Vol{{\rm Vol}}

\newtheorem{theorem}{Theorem}
\newtheorem{lemma}[theorem]{Lemma}

\newtheorem{proposition}[theorem]{Proposition}

\newcommand{\ritem}[1]{\item[{\rm #1}]}
\newenvironment{proof}{\noindent {\bf Proof:}}{$\Box$ \vspace{2 ex}}

\title{The density of discriminants of quintic rings and fields} 
\author{Manjul Bhargava} 

\begin{document}
\maketitle

\section{Introduction}

Let $N_n(X)$ denote the number of isomorphism classes of
number fields of degree $n$ having
absolute discriminant at most $X$.  Then it is an old folk conjecture
that the limit 
\vspace{.03in}
\begin{equation}\label{basiclimit}
c_n = \lim_{X\rightarrow\infty} \frac{N_n(X)}{X} \vspace{.035in}
\end{equation}
exists and is positive for $n>1$.  The conjecture is trivial for
$n\leq 2$, while for $n=3$ and $n=4$ it is a theorem of
Davenport and Heilbronn~\cite{DH} and of the author~\cite{Bhargava5},
respectively.  In degrees $n\geq 5$, where number fields tend to be
predominantly nonsolvable, the conjecture has not previously been
known to be true for any value of $n$.

The primary purpose of this article is to prove the above
conjecture for $n=5$.  In particular, we are able to determine the
constant $c_5$ explicitly.  More precisely, we prove:

\vspace{.015in}


\begin{theorem}
  Let $N_5^{(i)}(\xi,\eta)$ denote the number of quintic fields $K$,
  up to isomorphism,
  having $5-2i$ real embeddings and satisfying $\xi<\Disc(K)<\eta$.  Then
\[
\begin{array}{rlcl}\label{dodqf}
\rm{(a)}& \displaystyle{\lim_{X\rightarrow\infty} \frac{N_5^{(0)}(0,X)}{X}}
   &=&\!\! \displaystyle{\frac{1}{240}\prod_p (1+p^{-2}-p^{-4}-p^{-5})}; \\
\rm{(b)}& \displaystyle{\lim_{X\rightarrow\infty}
   \frac{N_5^{(1)}(-X,0)}{X}} &=&\!\!\!\;\;\displaystyle{\frac{1}{24}\,\prod_p
   (1+p^{-2}-p^{-4}-p^{-5})}; \\
\rm{(c)}&\displaystyle{\lim_{X\rightarrow\infty} \frac{N_5^{(2)}(0,X)}{X}} 
   &=& \!\!\!\;\;\displaystyle{\frac{1}{16}\,\prod_p (1+p^{-2}-p^{-4}-p^{-5})}.
\end{array}
\]
\vspace{-.1in}
\end{theorem}
The constants appearing in Theorem~1 (and thus their sum, $c_5
=\frac{13}{120}\,\prod_p (1+p^{-2}-p^{-4}-p^{-5})$) turn out to have
very natural interpretations.
Indeed, the constant $c_5$ takes the form of an Euler
product, where the Euler factor at a place $\nu$ ``counts'' the total number
of local \'etale quintic extensions 
of $\Q_\nu$, where
each isomorphism class of local extension $K_\nu$ is counted with a certain natural weight
to reflect the probability that a quintic number field $K$ has
localization $K\otimes \Q_\nu$ isomorphic to $K_\nu$ at $\nu$.
More precisely, let \vspace{.1in}
\begin{equation}\label{binfdef}
\beta_\infty = \frac{1}{2}\sum_{[K_\infty:\R]=5 \mbox{\small \,\,\'etale}}
\frac{1}{|\Aut_\R(K_\infty)|},
\end{equation}
where the sum is over all isomorphism classes $K_\infty$ 
of \'etale extensions of
$\R$ of degree 5.  Since $\Aut_\R(\R^5)=120$, 
$\Aut_\R(\R^3\oplus\C)=12$, and $\Aut_\R(\R\oplus\C^2)=8$, we have
$\beta_\infty = \frac{1}{240}+\frac1{24}+\frac1{16}=\frac{13}{120}$.  
Similarly, for each prime $p$, let
\begin{equation}\label{bpdef}
\beta_p = \frac{p-1}{p}\sum_{[K_p:\Q_p]=5 \mbox{\small \,\,\'etale}}
\frac{1}{|\Aut_{\Q_p}(K_p)|}\cdot\frac{1}{\Disc_p(K_p)},
\end{equation}
where the sum is over all isomorphism classes $K_p$ of \'etale extensions of
$\Q_p$ of degree 5, and $\Disc_p(K_p)$ denotes the discriminant of
$K_p$ viewed as a power of $p$.  Then
\begin{equation}
c_5 = \beta_\infty\cdot\prod_p \beta_p,
\end{equation}
since we will show that \begin{equation}\label{bpformula}
\beta_p=1+p^{-2}-p^{-4}-p^{-5}.\end{equation}  Thus we
obtain a natural interpretation of $c_5$ as a product of counts of
local field extensions.  
For more details on the evaluation of local sums of the form
(\ref{bpdef}),
and for global heuristics on the expected values of the asymptotic constants
associated to general 
degree $n$\, $S_n$-number fields,
see~\cite{Bhargava6}.

We obtain several additional results as by-products.
First, our methods enable us to analogously
count all {\it orders} in quintic fields:

\begin{theorem}
Let $M_5^{(i)}(\xi,\eta)$ denote the number of isomorphism classes of orders $\O$
in quintic fields having $5-2i$ real embeddings
and satisfying $\xi<\Disc(\O)<\eta$.  Then there exists a positive 
constant $\alpha$ such that
\[
\begin{array}{rlcc}\label{dodqr}
\rm{(a)}& \displaystyle{\lim_{X\rightarrow\infty} \frac{M_5^{(0)}(0,X)}{X}}
   &=&\!\! \displaystyle{\frac{\alpha}{240}}; \\
\rm{(b)}& \displaystyle{\lim_{X\rightarrow\infty}
   \frac{M_5^{(1)}(-X,0)}{X}} &=&\!\!
   \displaystyle{\frac{\alpha}{24}}; \\
\rm{(c)}&\displaystyle{\lim_{X\rightarrow\infty} \frac{M_5^{(2)}(0,X)}{X}} 
   &=& \!\!\displaystyle{\frac{\alpha}{16}}.
\end{array}
\]
\end{theorem}

The constant $\alpha$ in Theorem~2 has an analogous interpretation.  Let
$\alpha_p$ denote the analogue of the sum (\ref{bpdef}) for orders, i.e.,
\begin{equation}\label{gpdef}
\alpha_p = \frac{p-1}{p}\sum_{[R_p:\Z_p]=5}
\frac{1}{|\Aut_{\Z_p}(R_p)|}\cdot\frac{1}{\Disc_p(R_p)},
\end{equation}
where the sum is over all isomorphism classes of $\Z_p$-algebras $R_p$
of rank 5 over $\Z_p$ with nonzero discriminant.  Then we will show that
the constant $\alpha$ appearing in Theorem~2 is given by
\begin{equation}\label{rdef}
\alpha=\prod_p\alpha_p,
\end{equation}
thus expressing $\alpha$ 
as a product of counts of local ring extensions.  It
is an interesting combinatorial problem to explicitly evaluate
$\alpha_p$ in ``closed form'', analogous to the formula
(\ref{bpformula}) that we obtain for $\beta_p$;
see~\cite{Bhargava7} for some further discussion on the evaluation of
such sums.

Second, we note that the proof of Theorem~1 contains a determination of the
densities of the various splitting types of primes in $S_5$-quintic
fields.  If $K$ is an $S_5$-quintic field and $K_{120}$ denotes the
Galois closure of $K$, then the Artin symbol $(K_{120}/p)$ is defined
as a conjugacy class in $S_5$, its values being $\langle e \rangle$,
$\langle (12) \rangle$, $\langle (123) \rangle$, $\langle (1234)
\rangle$, $\langle (12345) \rangle$, $\langle (12)(34) \rangle$, or
$\langle (12)(345)\rangle$, where $\langle x\rangle$ denotes the
conjugacy class of $x$ in $S_5$.  It follows from the Cebotarev
density theorem that for fixed $K$ and varying $p$ (unramified in
$K$), the values $\langle e \rangle$, $\langle (12) \rangle$, $\langle
(123) \rangle$, $\langle (1234) \rangle$, $\langle (12345) \rangle$,
$\langle (12)(34) \rangle$, or $\langle (12)(345)\rangle$ occur with
relative frequency 1\,:\,10\,:\,20\,:\,30\,:\,24\,:\,15\,:\,20 (i.e.,
proportional to the size of the respective conjugacy class).  We
prove the following complement to Cebatorev density:

\begin{theorem}
Let $p$ be a fixed prime, and let $K$ run through all $S_5$-quintic
fields in which $p$ does not ramify, the fields being ordered by the size
of the discriminants.  Then the Artin symbol $(K_{120}/p)$ takes the values
$\langle e \rangle$,
$\langle (12) \rangle$, $\langle (123) \rangle$, $\langle (1234)
\rangle$, $\langle (12345) \rangle$, $\langle (12)(34) \rangle$, 
or $\langle (12)(345)\rangle$
with relative frequency 
$1\!:\!10\!:\!20\!:\!30\!:\!24\!:\!15\!:\!20$.  
\end{theorem}
Actually, we do a little more: we determine for each prime $p$ the
density of $S_5$-quintic fields $K$ in which $p$ has the various
possible ramification types.  For example, it follows from our
methods that
 a proportion of precisely
$\frac{(p+1)(p^2+p+1)}{p^4+p^3+2p^2+2p+1}$ of $S_5$-quintic fields are
ramified at $p$.

\vspace{.03in}

Lastly, our proof of Theorem 1 implies that nearly all---i.e., a
density of 100\% of---quintic fields have full Galois group $S_5$.
This is in stark contrast to the quartic case~\cite[Theorem
3]{Bhargava5}, where we showed that only about 91\% of quartic fields
have associated Galois group $S_4$:

\begin{theorem}
When ordered by absolute discriminant, a density of $100\%$
of quintic fields have associated Galois group $S_5$.  
\end{theorem}
In particular, it follows 
that 100\% of
quintic fields are nonsolvable.  


Note that, rather than counting quintic fields and orders up to
isomorphism, we could instead count these objects within a
fixed algebraic closure of $\Q$.  This would simply multiply all
constants appearing in Theorems 1 and 2 by five.  Meanwhile, 
Theorems 3 and 4 of course remain true regardless of whether one
counts quintic extensions up to isomorphism or within an algebraic closure of $\Q$.

The key ingredient that allows us to prove the above results for
quintic (and thus predominantly nonsolvable) fields is a
parametrization of isomorphism classes of quintic orders by means of
four integral alternating bilinear forms in five variables, up to the
action of $\GL_4(\Z)\times\SL_5(\Z)$, which we established
in~\cite{Bhargava4}.  The proofs of Theorems~1--4 can then be reduced
to counting appropriate integer points in certain fundamental regions,
as in~\cite{Bhargava5}.  However, the current case is considerably
more involved than the quartic case, since the relevant space is now
40-dimensional rather than 12-dimensional!  The primary difficulty
lies in counting points in the rather complicated cusps of these
40-dimensional fundamental regions (see
Lemmas~\ref{lem1}--\ref{hard}).
  
The necessary point-counting is accomplished in Section~2, by
carefully dissecting the ``irreducible'' portions of the fundamental regions
into 152 pieces, and then applying a new adaptation of the averaging
methods of
\cite{Bhargava5} to each piece (see Lemma~11).  The resulting counting
theorem (see Theorem~6), in conjunction with the results
of~\cite{Bhargava4}, then yields the asymptotic density of
discriminants of pairs $(R,R')$, where $R$ is an order in a quintic
field and $R'$ is a {\it sextic resolvent ring} of $R$.  Obtaining
Theorems 1--4 from this general density result then requires a sieve,
which in turn uses certain counting results on resolvent rings and
subrings obtained in \cite{Bhargava4} and in the recent work of
Brakenhoff~\cite{Jos}, respectively.  This sieve is carried out in the
final Section~3.


We note that the space of binary cubic forms that was used in the work of
Davenport-Heilbronn to count cubic fields, the space of pairs of ternary
quadratic forms that we used in~\cite{Bhargava5} to count quartic fields, 
and the space of
quadruples of alternating 2-forms in five variables that 
we use in this article, are all examples of what are
known as prehomogeneous vector spaces.  A {\it prehomogeneous vector space}
is a pair $(G,V)$, where $G$ is a reductive group and $V$ is a linear
representation of $G$ such that $G_\C$ has a Zariski open orbit on 
$V_\C$.  The concept was introduced by Sato in the 1960's and a
classification of all irreducible 
prehomogeneous vector spaces was given in the
work of Sato-Kimura~\cite{SatoKimura}, while
Sato-Shintani~\cite{SatoShintani} and Shintani~\cite{Shintani}
developed a theory of zeta functions
associated to these spaces.
\pagebreak

The connection between prehomogeneous vector spaces and field
extensions was first studied systematically in the beautiful 1992
paper of Wright-Yukie~\cite{WY}.  In this work, Wright and Yukie
determined the rational orbits and stabilizers in a number of
prehomogeneous vector spaces, and showed that these orbits correspond
to field extensions of degree 2, 3, 4, or 5.  In their paper, they laid
out a program to determine the density of discriminants of number
fields of degree up to five, by considering adelic versions of
Sato-Shintani's zeta functions as developed by Datskovsky and
Wright~\cite{DW} in their extensive work on cubic extensions.  

However, despite looking very promising, the program via adelic
Shintani zeta functions encountered some difficulties and has not
succeeded to date beyond the cubic case.  The primary difficulties
have been: (a) establishing cancellations among various divergent zeta
integrals, in order to establish a ``principal part formula'' for the
associated adelic Shintani zeta function; and (b) ``filtering'' out
the correct count of extensions from the overcount of extensions that
is inherent in the definition of the zeta function. In the quartic
case, difficulty (a) was overcome in the impressive 1995 treatise of
Yukie~\cite{Yukie}, while (b) remained an obstacle.  In the quintic
case, both (a) and (b) have remained impediments to obtaining a
correct count of quintic field extensions by discriminant.  (For more
on the Shintani adelic zeta function approach and these related
difficulties, see~\cite[\S1]{Bhargava5} and \cite{Yukie}.)

In~\cite{Bhargava5} and in the current article, we overcome the problems 
(a) and (b) above, for quartic and quintic fields respectively, by 
introducing a different counting method that relies more on 
geometry-of-numbers arguments.  Thus, although our methods are different, 
this article may be viewed as completing the program first laid out by 
Wright and Yukie~\cite{WY} to count field extensions in degrees up to 5 
via the use of appropriate prehomogeneous vector spaces.

\vspace{.05in} We now describe in more detail the methods of this
paper, and give a comparison with previous methods.  At least
initially, our approach to counting quintic extensions using the
prehomogeneous vector space $\C^4\otimes\wedge^2\C^5$ is quite similar
in spirit to Davenport-Heilbronn's original method in the cubic
case~\cite{DH} and its refinements developed in the quartic
case~\cite{Bhargava5}.  Namely, we begin by giving~an algebraic
interpretation of the {\it integer} orbits on the associated
prehomogeneous vector space which, in the quintic case, are the orbits
of the group $G_\Z=\GL_4(\Z)\times\SL_5(\Z)$ on the 
40-dimensional~lattice $V_\Z=\Z^4\otimes\wedge^2\Z^5$.  As we showed
in~\cite{Bhargava4}, these integer orbits have an extremely rich
algebraic~interpret-ation and structure (see Theorem~5 for a 
precise statement), enabling us to consider not only quintic fields,
but also more refined data such as all {\it orders} in quintic fields,
the local behaviors of these orders, and their sextic resolvent rings.
This interpretation of the integer orbits then allows us to reduce
our problem of counting orders and fields
to that of enumerating appropriate lattice
points in a fundamental domain for the action of the discrete group 
$G_\Z$ on the real vector space $V_\R=V_\Z\otimes\R$.  

Just as in \cite{DH} and \cite{Bhargava5}, the main difficulty in
counting lattice points in such a fundamental region is that this
region is {not} compact, but instead has cusps (or ``tentacles'')
going off to infinity.  To make matters even more interesting, unlike the
case of binary cubic forms in Davenport-Heilbronn's work---where there
is one relatively simple cusp defined by small degree inequalities in
four variables---in the case of quadruples of quinary alternating
2-forms, the cusps are numerous in number and are defined by
polynomial inequalities of extremely high degree in 40 variables!
These difficulties are further exacerbated by the fact that---contrary
to the cubic case---in the quartic and quintic cases the number of
nondegenerate lattice points in the cuspidal regions is of strictly
{\it greater} order than the number of points in the noncuspidal part
(``main body'') of the corresponding fundamental domains.  The latter
issue is indeed what lies behind the problems (a) and (b) above in the
adelic zeta function method.

Following our work in the quartic case~\cite{Bhargava5}, we overcome
these problems that arise from the cuspidal regions by counting
lattice points not in a single fundamental domain, but over a
{continuous, compact set of fundamental domains}.  This allows one to
``thicken'' the cusps, thereby gaining a good deal of control on the
integer points in these cuspidal regions.  A basic version of this
``averaging'' method was introduced and used in~\cite{Bhargava5} in
the quartic case to handle points in these cusps, and thus enumerate
quartic extensions by discriminant (see~\cite[\S1]{Bhargava5} for more
details).  
However, since the number, complexity and dimensions of the cuspidal
regions are so much greater in the quintic case than in the quartic
case, a number of new ideas and modifications are needed to
successfully carry out the same averaging method in the quintic case.

The primary technical contribution of this article is the introduction
of a method that allows one to systematically and canonically dissect
the cuspidal regions into certain ``nice'' subregions on which a
slightly refined averaging technique (see Sections~2.1--2.2) can then
be applied in a uniform manner.  Using this method, we divide up the
fundamental region into 159 pieces.  The first piece is the main body
of the region, where we show using geometry-of-numbers arguments that
the number of lattice points in the region is essentially its volume.
For each of the remaining 158 cuspidal pieces, we show, by a uniform
argument, that either the number of lattice points in that region is
negligible (see Table~1, Lemma~\ref{hard}), {\it or\,} that the
lattice points in that cuspidal piece are all {\it reducible}, i.e.,
they correspond to quintic rings that are not integral domains
(see~Lemma~\ref{red}).  An asymptotic formula for the number of
{irreducible} integer points in the entire fundamental domain is then
attained.  The interesting interaction between the algebraic
properties of the lattice points (via the correspondence
in~\cite{Bhargava4}) and their geometric locations within the
fundamental domain is therefore what allows us to overcome the
problems (a) and (b) arising in the adelic Shintani zeta function
method.  As explained earlier, a sieving method can then be used to
prove Theorems 1--4.

Our counting method in this article is quite robust and systematic,
and should be applicable in many other situations.  First, it can be
used to reprove the density of discriminants of cubic and quartic
fields, with much stronger error terms than have previously been known
(in fact, in the cubic case it can be used, in conjunction with a
sieve, to obtain an exact second order term; see~\cite{simpledhc}).
Second, the method can be suitably adapted to count cubic,
quartic, and quintic field extensions of any base number field
(see~\cite{dode}).  Third, the method can be used on prehomogeneous
vector spaces having {\it infinite} stabilizer groups, which would
also have a number of interesting applications (see, e.g., \cite{Bhargava8}).
Finally, we expect that the methods should also be adaptable to
representations of algebraic groups that are not necessarily
prehomogeneous.  We hope that these directions will be pursued further
in future~work.

\section{On the class numbers of quadruples of $5\times 5$ 
skew-symmetric matrices}

Let $V=V_\R$ denote the space of quadruples of $5\times 5$
skew-symmetric matrices over the real numbers.  We write an
element of $V_\R$ as an ordered quadruple $(A,B,C,D)$, where the
$5\times 5$ matrices $A$, $B$, $C$, $D$ have entries $a_{ij}$,
$b_{ij}$, $c_{ij}$, $d_{ij}$ respectively.  Such a quadruple
$(A,B,C,D)$ is said to be {\it integral} if all entries of the
matrices $A$, $B$, $C$, $D$ are integral.

The group $G_\Z=\GL_4(\Z)\times\SL_5(\Z)$ acts naturally on the space
$V_\R$.  Namely, an element $g_4\in\GL_4(\Z)$ acts by changing the
basis of the $\Z$-module of matrices spanned by $A,B,C,D$; in terms of
matrix multiplication, we have $(A\;B\;C\;D)^t\mapsto g_4\,
(A\;B\;C\;D)^t$. Similarly, an element $g_5\in \SL_5(\Z)$ changes
the basis of the five-dimensional space on which the skew-symmetric
forms $A,B,C,D$ take values, i.e., $g_5\cdot(A,B,C,D)=
(g_5Ag_5^t, g_5Bg_5^t, g_5 C g_5^t, g_5Dg_5^t)$.
It is clear that the actions of $g_4$ and $g_5$ commute, and that this
action of $G_\Z$ preserves the lattice $V_\Z$ consisting of the
integral elements of $V_\R$.

The action of $G_\Z$ on $V_\R$ (or $V_\Z$) has a unique polynomial
invariant, which we call the {\it discriminant}.  It is a degree 40
polynomial in 40 variables, and is much too large to write down.  An easy
method to compute it for any given element in $V$ was described
in~\cite{Bhargava4}.

The integer orbits of $G_\Z$ on $V_\Z$ have an important arithmetic
significance.  Recall that a {\it quintic ring} is any ring with unit
that is isomorphic to $\Z^5$ as a $\Z$-module; for example, an order in a
quintic number field is a quintic ring.  In~\cite{Bhargava4} we showed
how quintic rings may be parametrized in terms of the $G_\Z$-orbits on
$V_\Z$:

\begin{theorem}\label{main}
There is a canonical bijection between the set of $G_\Z$-equivalence
classes of elements $(A,B,C,D)\in V_\Z$, and the set of isomorphism
classes of pairs $(R,R')$, where $R$ is a quintic ring and $R'$ is a
sextic resolvent ring of $R$.  Under this bijection, we have
$\Disc(A,B,C,D)=\Disc(R)=\frac1{16}\cdot\Disc(R')^{1/3}$.
\end{theorem}
A {\it sextic resolvent} of a quintic ring $R$ is a sextic ring
$R'$ equipped with a certain {\it resolvent mapping} $R\to\wedge^2R'$
whose precise definition will not be needed here (see~\cite{Bhargava4}
for details).  In view of Theorem~\ref{main}, we wish to try and
understand the number of $G_\Z$-orbits on $V_\Z$ having absolute
discriminant at most~$X$, as $X\rightarrow\infty$.  The number of
integral orbits on $V_\Z$ having a fixed discriminant $\Delta$ is
called a ``class number'', and we wish to understand the behavior of
this class number on average.

From the point of view of Theorem~\ref{main}, we would like to
restrict the elements of $V_\Z$ under consideration to those that are
``irreducible'' in an appropriate sense.  More precisely, we call an element
$(A,B,C,D)\in V_\Z$ {\it irreducible} if, in the
corresponding pair of rings $(R,R')$ in Theorem~\ref{main}, the ring
$R$ is an integral domain.  The quotient field of $R$ is thus a
quintic field in that case.  
We say $(A,B,C,D)$ is {\it reducible} otherwise. 

One may also describe reducibility and 
irreducibility in more geometric terms.
If $(A,B,C,D)\in V_\Z$, then one may consider the $4\times 4$
sub-Pfaffians $Q_1(t_1,t_2,t_3,t_4),\ldots,Q_5(t_1,t_2,t_3,t_4)$ of
the single $5\times 5$ skew-symmetric matrix $At_1+Bt_2+Ct_3+Dt_4$
whose entries are linear forms in $t_1,t_2,t_3,t_4$.  In other words,
$Q_i=Q_i(w,x,y,z)$ is defined as a canonical squareroot of the
determinant of the $4\times 4$ matrix obtained from
$t_1A+t_2B+t_3C+t_4D$ by removing its $i$th row and column.  Thus
these $4\times 4$ Pfaffians $Q_1,\ldots,Q_5$ are quaternary quadratic
forms and so define five quadrics in $\P^3$.  If the element
$(A,B,C,D)\in V_\Z$ has nonzero discriminant, then it is known that
these five quadrics intersect in exactly five points in $\P^3$
(counting multiplicities); see e.g.,~\cite{WY}, \cite{Bhargava4}.  We
refer to these five points as the {\it zeroes of} $(A,B,C,D)$ in
$\P^3$.  In~\cite{Bhargava4} we showed that if $(A,B,C,D)$ corresponds
to $(R,R')$, where $R$ is isomorphic to an order in a quintic field
$K$, then there exists a zero of $(A,B,C,D)$ in $\P^3$ whose field of
definition is $K$.  (The other zeroes of $(A,B,C,D)\in V_\Z$ are thus
defined over the conjugates of $K$.)  Therefore, geometrically, we may
say that $(A,B,C,D)$ is irreducible if and only if it
possesses a zero in $\P^3$ having field of definition $K$, where $K$
is a quintic field extension of $\Q$.  On the other hand, $(A,B,C,D)$
is reducible if and only if $(A,B,C,D)$ possesses a zero in $\P^3$
defined over a number field of degree smaller than five.

The main result of this section is the following theorem: 

\begin{theorem}\label{cna}
Let $N(V^{(i)}_\Z;X)$ denote the number of $G_\Z$-equivalence classes
of irreducible elements $(A,B,C,D)\in V_\Z$ having $5-2i$ real 
zeroes in $\P^3$ and satisfying $|\Disc(A,B,C,D)|<X$.  Then 
\vspace{-.05in}
\[
\begin{array}{rlcl}\label{dodpotqf}
\rm{(a)}& \displaystyle{\lim_{X\rightarrow\infty} \frac{N(V^{(0)}_\Z;X)}{X}}
   &=&\!\! \displaystyle{\frac{\zeta(2)^2\zeta(3)^2\zeta(4)^2\zeta(5)}{240}}; \\[.1in]
\rm{(b)}& \displaystyle{\lim_{X\rightarrow\infty} \frac{N(V^{(1)}_\Z;X)}{X}}
   &=&\!\! \displaystyle{\frac{\zeta(2)^2\zeta(3)^2\zeta(4)^2\zeta(5)}{24}}; \\[.1in]
\rm{(c)}&\displaystyle{\lim_{X\rightarrow\infty} \frac{N(V^{(2)}_\Z;X)}{X}}
   &=&\!\! \displaystyle{\frac{\zeta(2)^2\zeta(3)^2\zeta(4)^2\zeta(5)}{16}}.
\end{array}
\]

\end{theorem}

Theorem~\ref{cna} is proven in several steps.  In Subsection 2.1, we
outline the necessary reduction theory needed to establish some
particularly useful fundamental domains for the action of $G_\Z$ on
$V_\R$.  In Subsections~2.2 and 2.3, we describe a refinement of 
the ``averaging'' method from
\cite{Bhargava5} that allows us to efficiently count integer points in
various components of these fundamental domains in terms of their
volumes.  In Subsections 2.4 and 2.5, we investigate the distribution of
reducible and irreducible integral points within these fundamental
domains.  The volumes of the resulting ``irreducible'' components of
these fundamental domains are then computed in Subsection~2.6, proving
Theorem~\ref{cna}.  A version of
Theorem~\ref{cna} for elements in $V_\Z$ satisfying any specified set
of congruence conditions  is then obtained in Subsection~2.7.

In Section 3, we will show how these counting methods---together
with a sieving argument---can be used to prove Theorems 1--4.

\subsection{Reduction theory}

The action of $G_\R=\GL_4(\R)\times \SL_5(\R)$ on $V_\R$ has three
nondegenerate orbits $V_\R^{(0)}, V_\R^{(1)}, V_\R^{(2)}$, where
$V^{(i)}_\R$ consists of those 
 elements $(A,B,C,D)$ in $V_\R$ having nonzero discriminant and $5-2i$
real zeroes in $\P^3$.  We wish to understand the number
of irreducible $G_\Z$-orbits on 
$V^{(i)}_\Z=V^{(i)}_\R\cap V_\Z$ 
having absolute discriminant at most $X$ ($i=0,1,2$).
We accomplish this by counting the number of integer points of 
absolute discriminant at most $X$ in suitable fundamental domains
for the action of $G_\Z$ on $V_\R$.

These fundamental regions are constructed as follows.  First, let $\FF$
denote a fundamental domain in $G_\R$ for $G_\Z\backslash G_\R$.  We may
assume that $\FF$ is contained in a standard Siegel set, 
i.e., we may assume $\FF$ is of the form $\FF=
\{nak\lambda:n\in N'(a),a\in A',k\in K,\lambda\in\Lambda\}$, where
\begin{eqnarray*}\label{siegel}
K\,&=&\{\mbox{special orthogonal transformations in $G_\R$}\};\\ 
A'&=&\{a(s_1,s_2,\ldots,s_7):s_1,s_2,\ldots,s_7\geq c\},\;\mbox{where}\,\\ 
{}&{}&a({\boldmath{s}})={\footnotesize
\left(\left(\begin{array}{cccc} \!\!\!s_1^{-3}s_2^{-1}s_3^{-1}
\!\!\!\!\!\!\!\!\!\!\! & 
{} & {} & {} \\[.04in] {} & \!\!\!\!\!\!\!\!\!s_1 s_2^{-1} s_3^{-1} \!\!\!\!\!\!\!\!
& {} &{} \\[.04in] {}&{}&\!\!\!\!\!\!\!\!s_1 s_2 s_3^{-1}\!\!\!\!\!\!\!\! 
&{}\\[.04in] {}&{}&{}&
\!\!\!\!\!\!\!\!s_1 s_2 s_3^{3}\!\!  \end{array}\right),
\left(\begin{array}{ccccc} \!\!s_4^{-4}s_5^{-3}s_6^{-2}s_7^{-1} \!\!\!\!\!\!\!\!
& {} & {} &{} &{}\\[.04in] 
{}& \!\!\!\!\!\!\!\!\!\!\!\!{}\!\!\!\!s_4s_5^{-3}s_6^{-2}s_7^{-1}\!\!\!\!\!\!\!\!&{} &{}&{}\\[.04in]{}&{}&
\!\!\!\!\!\!\!\!\!\!\!{}\!\!s_4 s_5^{2}s_6^{-2}s_7^{-1}\!\!\!\!\!\!\!\!&{}&
{}\\[.04in]
{}&{}&{}& \!\!\!\!\!\!\!\!\!s_4s_5^{2}s_6^{3}s_7^{-1}\!\!\!\!\!\!\!\! 
&{}\\[.04in] {}&{}&{}&{}&\!\!\!\!\!\!\!\!s_4s_5^{2}s_6^{3}s_7^4\!\! \end{array} \right)\right)};\,\,\\
\bar N'\,&=&\{n(u_1,u_2,\ldots,u_{16}):u=(u_1,u_2,\ldots,u_{16})\in\nu(a) \},\;\mbox{where}\,\\ 
{}&{}&
n({\boldmath{u}})={\footnotesize
\left(\left(\begin{array}{cccc} 1         & {} & {} & {} \\ {u_1} & 1 
& {} &{} \\ {u_2}&{u_3}&1&{}\\{u_4}&{u_5}&{u_6}&1 \end{array}\right),
\left(\begin{array}{ccccc} 1 & {} & {} &{} &{}\\ 
{u_7}& 1 &{} &{}&{}\\{u_8}&{u_9}&
1&{}&{}\\{u_{10}}&{u_{11}}&{u_{12}}&1&{}\\ 
{u_{13}}&{u_{14}}&{u_{15}}&{u_{16}}&1\end{array} \right)\right)};\,\, \\
\Lambda\,&=&\{\{\lambda:\lambda>0\},\;\mbox{where}\,\\
{}&{}& \lambda \mbox{ acts by }{\footnotesize
\left(\left(\begin{array}{cccc} \lambda & {} & {} & {} \\ {} & \lambda 
& {} &{} \\ {}&{}&\lambda&{}\\{}&{}&{}&\lambda \end{array}\right),
\left(\begin{array}{ccccc} 1 & {} & {} &{} &{}\\ 
{}& 1 &{} &{}&{}\\{}&{}&
1&{}&{}\\{}&{}&{}&1&{}\\ 
{}&{}&{}&{}&1\end{array} \right)\right)};
\end{eqnarray*}
here $c>0$ is an absolute constant and $\nu(a)$ is
an absolutely bounded measurable subset of $\R^{16}$ 
dependent only on the value of $a\in A'$.


For $i=0,1,2$, let $n_i$ denote the cardinality of the stabilizer in
$G_\R$ of any element $v\in V^{(i)}_\R$ (it follows from
Proposition~\ref{covering} below that $n_1=120$, $n_2=12$, and $n_3=8$).
Then for any $v\in V^{(i)}_\R$, $\FF v$ will be the union of $n_i$
fundamental domains for the action of $G_\Z$ on $V^{(i)}_\R$.  Since
this union is not necessarily disjoint, $\FF v$ is best viewed as a
multiset, where the multiplicity of a point $x$ in $\FF v$ is given by
the cardinality of the set $\{g\in\FF\,\,|\,\,gv=x\}$. Evidently, this
multiplicity is a number between 1 and $n_i$.

Even though the multiset $\FF v$ is the union of $n_i$ fundamental
domains for the action of $G_\Z$ on $V^{(i)}_\R$, not all elements in
$G_\Z\backslash V_\Z$ will be represented in $\FF v$ exactly $n_i$
times.  In general, the number of times the $G_\Z$-equivalence class
of an element $x\in V_\Z$ will occur in $\FF v$ is given by
$n_i/m(x)$, where $m(x)$ denotes the size of the stabilizer of $x$ in
$G_\Z$.  We define $N(V_\Z^{(i)};X)$ to be the (weighted) number of
irreducible $G_\Z$-orbits on $V_\Z^{(i)}$ having absolute discriminant
at most $X$, where each orbit is counted by a weight of $1/m(x)$ for
any point $x$ in that orbit.  Thus $n_i\cdot N(V_\Z^{(i)};X)$ is the
(weighted) number of points in $\FF v$ 
having absolute discriminant at most $X$, where each point $x$ in the
multiset $\FF v$ is counted with a weight of $1/m(x)$.

We note that the $G_\Z$-orbits in $V_\Z$ corresponding to orders in
non-Galois quintic fields will then each be counted simply with a
weight of 1, since such orders can have no automorphisms.  We will
show (see Lemma~\ref{3reducible2}) that orbits having weight $< 1$
are negligible in number in comparison to those having weight $1$, and
so points of weight $<1$ will not be important as they will not
affect the main term of the asymptotics of $N(V_\Z^{(i)}; X)$ as
$X\to\infty$.


Now the number of integer points can be difficult to count in a
single fundamental region $\FF v$.  The main technical obstacle
is that the fundamental region $\FF v$ is not compact, but rather has
a system of 
cusps going off to infinity which in fact contains infinitely many
points, including many irreducible points.  We simplify the counting
of such points by ``thickening'' the cusp; more precisely, we compute
the number of points in the fundamental region $\FF v$ by averaging
over lots of such fundamental domains, i.e., by averaging over a
continuous range of points $v$ lying in a certain special compact
subset $H$ of $V$.

\subsection{Averaging over fundamental domains}


Let $H=H(J)=\{w\in V : \|w\|\leq J,\;|\Disc(w)|\geq 1\}$, where
$\|w\|$ denotes a Euclidean norm on $V$ fixed under the action of $K$,
and $J$ is sufficiently large so that $H$ is nonempty and of nonzero
volume.  
We write $V^{(i)}:=V^{(i)}_\R$.  Then we have
\begin{equation}
N(V^{(i)}_\Z;X) = \frac{\int_{v\in H\cap V^{(i)}} 
\#\{x\in \FF v\cap V_\Z^\irr: |\Disc(x)|<X\}\;
|\Disc(v)|^{-1} dv}
{n_i\cdot\int_{v\in H\cap V^{(i)}} \:|\Disc(v)|^{-1} dv},
\end{equation}
where $V_\Z^\irr\subset V_\Z$ denotes the subset of 
irreducible points in $V_\Z$.  The denominator of the latter
expression is, by construction, a finite absolute constant
$M_i=M_i(J)$ greater than zero.  We have chosen the measure
$|\Disc(v)|^{-1}\,dv$ because it is a $G_\R$-invariant measure.

More generally, for any $G_\Z$-invariant subset $S\subset
V_\Z^{(i)}$, let $N(S;X)$ denote the number of irreducible
$G_\Z$-orbits on $S$ having discriminant less than $X$.  Then
$N(S;X)$ can be expressed as
\begin{equation}\label{nsx}
N(S;X) = \frac{\int_{v\in H\cap V^{(i)}} 
\#\{x\in \FF v\cap S^\irr: |\Disc(x)|<X\}\;
|\Disc(v)|^{-1} dv}
{n_i\cdot\int_{v\in H\cap V^{(i)}} \:|\Disc(v)|^{-1} dv},
\end{equation}
where $S^\irr\subset S$ denotes the subset of irreducible
points in $S$.  We shall use this definition of $N(S;X)$ for any
$S\subset V_\Z$, even if $S$ is not $G_\Z$-invariant.  Note that for
disjoint $S_1,S_2\subset V_\Z$, we have $N(S_1\cup S_2)=N(S_1)+N(S_2)$.

Now since $|\Disc(v)|^{-1}\,dv$ is a $G_\R$-invariant measure, we have
for any $f\in C_0(V^{(i)})$, with $v,x\in V_\R^{(i)}$ and $g\in G_\R$
satisfying $v=gx$, that $f(v)|\Disc(v)|^{-1} dv= r_i\, f(gx)\,dg$ for
some constant $r_i$ dependent only on whether $i=0$, $1$ or $2$; here
$dg$ denotes a left-invariant Haar measure on $G_\R$.  We may thus
express the above formula for $N(S;X)$ as an integral over $\FF\subset
G_\R$:
\begin{eqnarray}
N(S;X)&\!\!=\!\!&  \frac{r_i}{M_i}\int_{g\in\FF}
\#\{x\in S^\irr\cap gH:|\Disc(x)|<X\}\,dg\\[.075in] 
&\!\!=\!\!&  \frac{r_i}{M_i}\int_{g\in N'(a)A'\Lambda K}
\#\{x\in S\cap \bar n(u)a(s) \lambda k H:|\Disc(x)|<X\}\,dg\,.
\end{eqnarray}
Let us write $H(u,s,\lambda,X) = 
\bar n(u)a(s)\lambda H\cap\{v\in V^{(i)}:|\Disc(v)|<X\}$.
Noting that $KH=H$, $\int_K dk = 1$ (by convention), and
$dg =  s_1^{-12}s_2^{-8}s_3^{-12}s_4^{-20}s_5^{-30}s_6^{-30}s_7^{-20}
du\,d^\times s\, d^\times\lambda\, dk$ (up to scaling), we have
\begin{equation}\label{avg}
N(S;X) = \frac{r_i}{M_i}\int_{g\in N'(a)A'\Lambda}                              
\#\{x\in S^\irr\cap H(u,s,\lambda,X)\}
\,s_1^{-12}s_2^{-8}s_3^{-12}s_4^{-20}s_5^{-30}s_6^{-30}s_7^{-20} \,
du\, d^\times t\,d^\times \lambda\,.
\end{equation}

We note that the same counting method may be used even if we are interested
in counting both reducible and irreducible orbits in $V_\Z$.  For any
set $S\subset V_\Z^{(i)}$, 
let $N^*(S;X)$ be defined by (\ref{nsx}), but where the 
superscript ``irr'' is removed.  Thus for a 
$G_\Z$-invariant set $S\subset V_\Z^{(i)}$, \,$n_i\cdot N^*(S;X)$ counts the
total (weighted) number
of $G_\Z$-orbits in $S$ having absolute discriminant 
nonzero and less than $X$ (not just the irreducible ones).  
By the same reasoning, we have
\begin{equation}\label{avgS}
N^*(S;X) = \frac{r_i}{M_i}\int_{g\in N'(a)A'\Lambda}                              
\#\{x\in S\cap H(u,s,\lambda,X)\}
\, s_1^{-12}s_2^{-8}s_3^{-12}s_4^{-20}s_5^{-30}s_6^{-30}s_7^{-20} \,
du\, d^\times t\,d^\times \lambda\,.
\end{equation}
The expression (\ref{avg}) for $N(S;X)$, and its analogue
(\ref{avgS}) for $N^*(S,X)$,
will be useful in the sections that follow.

\subsection{A lemma from geometry of numbers}

To estimate the number of lattice points in $H(u,s,\lambda,X)$, we
have the following elementary proposition from the
geometry-of-numbers, which is essentially due to
Davenport~\cite{Davenport1}. 
To state the proposition, we require the following simple definitions.
A multiset $\mathcal R\subset\R^n$ is said to be {\it measurable} if
$\mathcal R_k$ is measurable for all $k$, where $\mathcal R_k$ denotes
the set of those points in $\mathcal R$ having a fixed multiplicity
$k$.  Given a measurable multiset $\mathcal R \subset\R^n$, we define
its volume in the natural way, that is, $\Vol(\mathcal R)=\sum_k
k\cdot\Vol(\mathcal R_k)$, where $\Vol(\mathcal R_k)$ denotes the
usual Euclidean volume of $\mathcal R_k$.

\begin{lemma}\label{genbound}
  Let $\mathcal R$ be a bounded, semi-algebraic multiset in $\R^n$
  having maximum multiplicity $m$, and which is defined by at most $k$
  polynomial inequalities each having degree at most $\ell$.  Let $\RR'$
  denote the image of $\RR$ under any $($upper or lower$)$ triangular,
  unipotent transformation of $\R^n$.  Then the number of integer
  lattice points $($counted with multiplicity$)$ contained in the
  region $\mathcal R'$ is
\[\Vol(\mathcal R)+ O(\max\{\Vol(\bar{\mathcal R}),1\}),\]
where $\Vol(\bar{\mathcal R})$ denotes the greatest $d$-dimensional 
volume of any projection of $\mathcal R$ onto a coordinate subspace
obtained by equating $n-d$ coordinates to zero, where 
$d$ takes all values from
$1$ to $n-1$.  The implied constant in the second summand depends
only on $n$, $m$, $k$, and $\ell$.
\end{lemma}
Although Davenport states the above lemma only for compact
semi-algebraic sets $\mathcal R\subset\R^n$, his proof adapts without
essential change to the more general case of a bounded semi-algebraic
multiset $\RR\subset\R^n$, with the same estimate applying also to
any image $\mathcal R'$ of $\mathcal R$ under a unipotent triangular
transformation.



\subsection{Estimates on reducible quadruples $(A,B,C,D)$}

In this section we describe the relative frequencies with which
reducible and irreducible elements sit inside various parts of the
fundamental domain $\FF v$, as $v$ varies over the compact region $H$.

We begin by describing some sufficient conditions that guarantee that
a point in $V_\Z$ is reducible.

\begin{lemma}\label{lem1}
  Let $(A,B,C,D)\in V_\Z$ be an element such that some non-trivial 
  $\Q$-linear combination of $A,B,C,D$ has rank $\leq 2$.  Then $(A,B,C,D)$
  is reducible.
\end{lemma}

\begin{proof}
  Suppose $E=rA+sB+tC+uD$, where $r,s,t,u\in\Q$ are not all 
  zero.  Let $Q_1,\ldots,Q_5$ denote the five $4\times 4$
  sub-Pfaffians of $(A,B,C,D)$.  Then we have proven in
  \cite{Bhargava4} that if $(A,B,C,D)\in V_\Z$ is irreducible, then
  the quadrics $Q_1=0,\ldots,Q_5=0$ intersect in five points in
  $\P^3(\bar\Q)$, and moreover, these five points are defined over
  conjugate quintic extensions of $\Q$.  However, if rank$(E)\leq 2$,
  then $[r,s,t,u]\in\P^3(\Q)$ is a common zero of $Q_1,\ldots,Q_5$ and
  it is defined over $\Q$, contradicting the irreducibility of
  $(A,B,C,D)$.
\end{proof}

\begin{lemma}\label{lem2}
  Let $(A,B,C,D)\in V_\Z$ be an element such that some non-trivial 
  $\Q$-linear combination of $Q_1,\ldots,Q_5$ factors over
  $\Q$ into two linear factors, where $Q_1,\ldots,Q_5$ denote the five
  $4\times 4$ sub-Pfaffians of $(A,B,C,D)$.  Then $(A,B,C,D)$ is 
  reducible.
\end{lemma}

\begin{proof}
  As noted in the proof of Lemma~\ref{lem1}, the five associated
  quadratic forms $Q_1,\ldots,Q_5$ of an irreducible element
  $(A,B,C,D)\in V_\Z$ possess five common zeroes that are defined over
  conjugate quintic fields, and these zeroes are conjugate to each
  other over $\Q$.  It follows that each of the ${5\choose 3}=10$
  planes, going through subsets of three of those five points, cannot
  be defined over $\Q$, as these planes will each be part of a
  $\Gal(\bar\Q/\Q)$-orbit of size at least 5.  However, if some
  rational quaternary quadratic form $Q$ factors over $\Q$ into linear
  factors, then (by the pigeonhole principle) at least one of these
  two rational factors must vanish at three of the five common points
  of intersection, a contradiction.
\end{proof} 

\begin{lemma}\label{red}
Let $(A,B,C,D)\in V_\Z$ be an element such that all the variables
in at least one of the following sets vanish:
\begin{itemize}
\ritem{(i)}   $\{a_{12},a_{13},a_{14},a_{15},a_{23},a_{24},a_{25}\}$
\ritem{(ii)}  $\{a_{12},a_{13},a_{14},a_{23},a_{24},a_{34}\}$ 
\ritem{\;\;\;(iii)} 
$\{a_{12},a_{13},a_{14},a_{15}\}\cup\{b_{12},b_{13},b_{14},b_{15}\} $
\ritem{(iv)} 
$\{a_{12},a_{13},a_{14},a_{23},a_{24}\}\cup\{b_{12},
b_{13},b_{14},b_{23},b_{24}\} $
\ritem{(v)}  
$\{a_{12},a_{13},a_{14}\}\cup\{b_{12},b_{13},b_{14}\}\cup\{c_{12},
c_{13},c_{14}\}$
\ritem{(vi)}  
$\{a_{12},a_{13},a_{23}\}\cup\{b_{12},b_{13},b_{23}\}\cup\{c_{12},
c_{13},c_{23}\}$
\ritem{(vii)}
$\{a_{12},a_{13}\}\cup\{b_{12},b_{13}\}\cup\{c_{12},c_{13}\}\cup\{d_{12},
d_{13}\}$
\end{itemize} \vspace{.045in}
Then $(A,B,C,D)$ is reducible.
\end{lemma}

\begin{proof}
  In cases (i) and (ii), one sees that $A$ has rank $\leq 2$, and thus
  $(A,B,C,D)$ is reducible by Lemma~\ref{lem1}.
  In the remaining cases (iii)--(vii), one finds that $Q_5$ factors
  into rational linear factors, and thus the result in these cases
  follows from Lemma~\ref{lem2}.
\end{proof}

We are now ready to give an estimate on the number of 
irreducible elements in $\FF v$, on average, satisfying
$a_{12}=0$:

\begin{lemma}\label{hard}
Let $v$ take a random value in $H$ uniformly with respect to the measure
$|\Disc(v)|^{-1}\,dv$.  Then the expected number of 
irreducible elements $(A,B,C,D)\in\FF v$ such that 
$|\Disc(A,B,C,D)|< X$  and $a_{12}=0$
is $O(X^{39/40})$.
\end{lemma}

\begin{proof}
As in~\cite{Bhargava5}, we divide the set of all $(A,B,C,D)\in V_\Z$
into a number of cases depending on which initial coordinates are zero
and which are nonzero.  These cases are described in the second column
of Table 1.  The vanishing conditions in the various subcases of
Case~$n+1$ are obtained by setting equal to 0---one at a time---each
variable that was assumed to be nonzero in Case~$n$.  If such a
resulting subcase satisfies the reducibility conditions of
Lemma~\ref{red}, it is not listed.  In this way, it becomes clear that
any irreducible element in $V_\Z$ must satisfy precisely
one of the conditions enumerated in the second column of Table 1.
In particular, there is no Case 14, because assuming any
nonzero variables in Case 13 to be zero immediately results in
reducibility by Lemma~\ref{red}.

Let $T$ denote the set of all forty variables
$a_{ij},b_{ij},c_{ij},d_{ij}$.  For a subcase $\CC$ of Table~1, we use
$T_0=T_0(\CC)$ to denote the set of variables in $T$ assumed to be 0
in Subcase~$\CC$, and $T_1$ to denote the set of variables in $T$
assumed to be nonzero.  

Each variable $t\in T$ has a {\it weight}, defined as follows. The
action of $a(s_1,s_2,\ldots,s_7)\cdot\lambda$ on $(A,B,C,D)\in
V$ causes each variable $t$ to multiply by a certain weight which we
denote by $w(t)$.  These weights $w(t)$ are evidently rational
functions in $\lambda,s_1,\ldots,s_7$.

Let $V(\CC)$ denote the set of $(A,B,C,D)\in V_\R$ such that
$(A,B,C,D)$ satisfies the vanishing and nonvanishing conditions of
Subcase~$\CC$.  For example, in Subcase~2a we have
$T_0(\mbox{2a})=\{a_{12},a_{13}\}$ and
$T_1(\mbox{2a})=\{a_{14},a_{23},b_{12}\}$; thus $V(\mbox{2a})$
denotes the set of all $(A,B,C,D)\in V_\Z$ such that $a_{12}=a_{13}=0$
but $a_{14},a_{23},b_{12}\neq 0$.

For each subcase $\CC$ of Case $n$ ($n>0$), we wish to show that
$N(V(\CC);X)$, as defined by (\ref{nsx}), is $O(X^{39/40})$.  
Since $N'(a)$ is absolutely bounded, the equality (\ref{avgS}) implies that
\begin{equation}\label{estv0s}
N^*(V(\CC);X)\ll
\int_{\lambda=c'}^{X^{1/40}} \int_{s_1,s_2,\ldots,s_7=c}^\infty
\sigma(V(\CC))
\, s_1^{-12}s_2^{-8}s_3^{-12}s_4^{-20}s_5^{-30}s_6^{-30}s_7^{-20} \,
d^\times\! s \,d^\times\!\lambda,
\end{equation}
where $\sigma(V(\CC))$ denotes the number of integer points
in the region $H(u,s,\lambda,X)$ that also satisfy the conditions 
\begin{equation}\label{cond}
\mbox{$t=0$ for $t\in T_0$ and $|t|\geq 1$ for $t\in T_1$}.
\end{equation}

Now for an element $(A,B,C,D)\in H(u,s,\lambda,X)$, we evidently have
\begin{equation}\label{condt}
|t|\leq J{w(t)}
\end{equation}
and therefore the number of integer points in $H(u,s,\lambda,X)$
satisfying (\ref{cond}) will be nonzero only if we have
\begin{equation}\label{condt1}
J{w(t)}\geq 1
\end{equation}
for all weights $w(t)$ such that $t\in T_1$.  Now the sets $T_1$ in
each subcase of Table~1 have been chosen to be precisely the set of
variables having the minimal weights $w(t)$ among the variables $t\in
T\setminus T_0$ (by ``minimal weight'' in $T\setminus T_0$, we mean
there is no other variable $t\in T\setminus T_0$ with weight having
smaller exponents for all parameters $\lambda, s_1,s_2,\ldots,s_7$).
Thus if the condition (\ref{condt1}) holds for all weights $w(t)$
corresponding to $t\in T_1$, then---by the very choice of $T_1$---we
will also have $Jw(t)\gg 1$ for all weights $w(t)$ such that $t\in
T\setminus T_0$.


Therefore, if the region $\H=\{(A,B,C,D)\in
H(u,s,\lambda,X):t=0\;\;\forall t\in T_0;\;\; |t|\geq 1\;\; \forall
t\in T_1\}$ contains an integer point, then (\ref{condt1}) and
Lemma~\ref{genbound} together imply that the number of integer points
in $\H$ is $O(\Vol(\H))$, since the volumes of all the projections of
$u^{-1}\H$ will in that case also be $O(\Vol(\H))$.  
Now clearly
\[\Vol(\H)=O\Bigl(J^{40-|T_0|}\prod_{t\in T\setminus T_0} w(t)\Bigr),\]
so we obtain
\begin{equation}\label{estv1s}
N(V(\CC);X)\ll
\int_{\lambda=c'}^{X^{1/40}} \int_{s_1,s_2,\ldots,s_7=c}^\infty
\prod_{t\in T\setminus T_0} w(t)
\,\, s_1^{-12}s_2^{-8}s_3^{-12}s_4^{-20}s_5^{-30}s_6^{-30}s_7^{-20} \,\,\,
d^\times\! s \,d^\times\!\lambda.
\end{equation}

The latter integral can be explicitly carried out for each of the
subcases in Table~1.  It will suffice, however, to have a simple
estimate of the form $O(X^r)$, with $r<1$, for the integral
corresponding to each subcase.  For example, if the total exponent of
$s_i$ in (\ref{estv1s}) is negative for all $i$ in $\{1,\ldots,7\}$,
then it is clear that the resulting integral will be at most
$O(X^{(40-|T_0|)/40})$ in value.  This condition holds for many of the
subcases in Table 1 (indicated in the fourth column by ``-''),
immediately yielding the estimates given in the third column.

For cases where this negative exponent condition does not hold, the
estimate given in the third column can be obtained as follows.  The
factor $\pi$ given in the fourth column is a product of variables in
$T_1$, and so it is at least one in absolute value.  The integrand in
(\ref{estv1s}) may thus be multiplied by $\pi$ without harm, and the
estimate (\ref{estv1s}) will remain true; we may then apply the
inequalities (\ref{condt}) to each of the variables in $\pi$, yielding
\begin{equation}\label{estv2s}
N(V(\CC);X)\ll
\int_{\lambda=c'}^{X^{1/40}} \int_{s_1,s_2,\ldots,s_7=c}^\infty
\prod_{t\in T\setminus T_0} w(t)\;w(\pi)
\,\, s_1^{-12}s_2^{-8}s_3^{-12}s_4^{-20}s_5^{-30}s_6^{-30}s_7^{-20} \,\,\,
d^\times\! s \,d^\times\!\lambda.
\end{equation}
where we extend the notation $w$ multiplicatively, i.e.,
$w(ab)=w(a)w(b)$.  In each subcase of Table~1, we have chosen the
factor $\pi$ so that the total exponent of each $s_i$ in
(\ref{estv2s}) is negative.  Thus we obtain from (\ref{estv2s}) that
$N(V(\CC);X)=O(X^{(40-\#T_0(\CC)+\#\pi)/40}$, where $\#\pi$ denotes the
total number of variables of $T$ appearing in $\pi$ (counted with
multiplicity), and this is precisely the estimate given in the third
column of Table~1.  In every subcase, aside from Case~0, we see that
$40-\#T_0+\#\pi<40$, as desired.
\end{proof}

\begin{tabular}{|r|l|c|c|}\hline
Case & The set $S\subset V_\Z$ defined by & $N(S;X)\ll$ & Use factor \\ 
\hline\hline 
%
%
0. & ${a_{12}}\neq0\,$ & $X^{40/40}$ &  -  \\
\hline
1. & ${a_{12}}=0\,;$ & $X^{39/40}$ &  -  \\
  & ${a_{13}, b_{12}}\neq 0$ &  &  \\ \hline
2a. & ${a_{12}, a_{13}}=0\,;$ & $X^{38/40}$ &  -  \\
  & ${a_{14}, a_{23}, b_{12}}\neq 0$ &  &  \\ \hline
2b. & ${a_{12}, b_{12}}=0\,;$ & $X^{38/40}$ &  -  \\
  & ${a_{13}, c_{12}}\neq 0$ &  &  \\ \hline
3a. & ${a_{12}, a_{13}, a_{14}}=0\,;$ & $X^{37/40}$ &  -  \\
  & ${a_{15}, a_{23}, b_{12}}\neq 0$ &  &  \\ \hline
3b. & ${a_{12}, a_{13}, a_{23}}=0\,;$ & $X^{37/40}$ &  -  \\
  & ${a_{14}, b_{12}}\neq 0$ &  &  \\ \hline
3c. & ${a_{12}, a_{13}, b_{12}}=0\,;$ & $X^{37/40}$ &  -  \\
  & ${a_{14}, a_{23}, b_{13}, c_{12}}\neq 0$ &  &  \\ \hline
3d. & ${a_{12}, b_{12}, c_{12}}=0\,;$ & $X^{37/40}$ &  -  \\
  & ${a_{13}, d_{12}}\neq 0$ &  &  \\ \hline
4a. & ${a_{12}, a_{13}, a_{14}, a_{15}}=0\,;$ & $X^{37/40}$ & $ a_{23}$ \\
  & ${a_{23}, b_{12}}\neq 0$ &  &  \\ \hline
4b. & ${a_{12}, a_{13}, a_{14}, a_{23}}=0\,;$ & $X^{37/40}$ & $ a_{24}$ \\
  & ${a_{15}, a_{24}, b_{12}}\neq 0$ &  &  \\ \hline
4c. & ${a_{12}, a_{13}, a_{14}, b_{12}}=0\,;$ & $X^{36/40}$ &  -  \\
  & ${a_{15}, a_{23}, b_{13}, c_{12}}\neq 0$ &  &  \\ \hline
4d. & ${a_{12}, a_{13}, a_{23}, b_{12}}=0\,;$ & $X^{36/40}$ &  -  \\
  & ${a_{14}, b_{13}, c_{12}}\neq 0$ &  &  \\ \hline
4e. & ${a_{12}, a_{13}, b_{12}, b_{13}}=0\,;$ & $X^{36/40}$ &  -  \\
  & ${a_{14}, a_{23}, c_{12}}\neq 0$ &  &  \\ \hline
4f. & ${a_{12}, a_{13}, b_{12}, c_{12}}=0\,;$ & $X^{36/40}$ &  -  \\
  & ${a_{14}, a_{23}, b_{13}, d_{12}}\neq 0$ &  &  \\ \hline
4g. & ${a_{12}, b_{12}, c_{12}, d_{12}}=0\,;$ & $X^{36/40}$ &  -  \\
  & ${a_{13}}\neq 0$ &  &  \\ \hline
5a. & ${a_{12}, a_{13}, a_{14}, a_{15}, a_{23}}=0\,;$ & $X^{37/40}$ &
    $ a_{24}^2 $ \\
  & ${a_{24}, b_{12}}\neq 0$ &  &  \\ \hline
5b. & ${a_{12}, a_{13}, a_{14}, a_{15}, b_{12}}=0\,;$ & $X^{35/40}$ &  -  \\
  & ${a_{23}, b_{13}, c_{12}}\neq 0$ &  &  \\ \hline
5c. & ${a_{12}, a_{13}, a_{14}, a_{23}, a_{24}}=0\,;$ & $X^{37/40}$ &
    $ a_{34}^2 $ \\
  & ${a_{15}, a_{34}, b_{12}}\neq 0$ &  &  \\ \hline
5d. & ${a_{12}, a_{13}, a_{14}, a_{23}, b_{12}}=0\,;$ & $X^{35/40}$ &  -  \\
  & ${a_{15}, a_{24}, b_{13}, c_{12}}\neq 0$ &  &  \\ \hline
\end{tabular}

\vspace{.15in}
\begin{center}
{\bf Table 1.} Subcases 0--5d.
\end{center}

\begin{tabular}{|r|l|c|c|}\hline
Case & The set $S\subset V_\Z$ defined by & $N(S;X)\ll$ & Use factor \\ 
\hline\hline 
5e. & ${a_{12}, a_{13}, a_{14}, b_{12}, b_{13}}=0\,;$ & $X^{35/40}$ &  -  \\
  & ${a_{15}, a_{23}, b_{14}, c_{12}}\neq 0$ &  &  \\ \hline
5f. & ${a_{12}, a_{13}, a_{14}, b_{12}, c_{12}}=0\,;$ & $X^{35/40}$ &  -  \\
  & ${a_{15}, a_{23}, b_{13}, d_{12}}\neq 0$ &  &  \\ \hline
5g. & ${a_{12}, a_{13}, a_{23}, b_{12}, b_{13}}=0\,;$ & $X^{35/40}$ &  -  \\
  & ${a_{14}, b_{23}, c_{12}}\neq 0$ &  &  \\ \hline
5h. & ${a_{12}, a_{13}, a_{23}, b_{12}, c_{12}}=0\,;$ & $X^{35/40}$ &  -  \\
  & ${a_{14}, b_{13}, d_{12}}\neq 0$ &  &  \\ \hline
5i. & ${a_{12}, a_{13}, b_{12}, b_{13}, c_{12}}=0\,;$ & $X^{35/40}$ &  -  \\
  & ${a_{14}, a_{23}, c_{13}, d_{12}}\neq 0$ &  &  \\ \hline
5j. & ${a_{12}, a_{13}, b_{12}, c_{12}, d_{12}}=0\,;$ & $X^{35/40}$ &  -  \\
  & ${a_{14}, a_{23}, b_{13}}\neq 0$ &  &  \\ \hline
6a. & ${a_{12}, a_{13}, a_{14}, a_{15}, a_{23}, a_{24}}=0\,;$ & $X^{37/40}$ &
    $ a_{34}^3 $ \\
  & ${a_{25}, a_{34}, b_{12}}\neq 0$ &  &  \\ \hline
6b. & ${a_{12}, a_{13}, a_{14}, a_{15}, a_{23}, b_{12}}=0\,;$ & $X^{35/40}$ &
    $ a_{24}$ \\
  & ${a_{24}, b_{13}, c_{12}}\neq 0$ &  &  \\ \hline
6c. & ${a_{12}, a_{13}, a_{14}, a_{15}, b_{12}, b_{13}}=0\,;$ & $X^{34/40}$ &
     -  \\
  & ${a_{23}, b_{14}, c_{12}}\neq 0$ &  &  \\ \hline
6d. & ${a_{12}, a_{13}, a_{14}, a_{15}, b_{12}, c_{12}}=0\,;$ & $X^{34/40}$ &
     -  \\
  & ${a_{23}, b_{13}, d_{12}}\neq 0$ &  &  \\ \hline
6e. & ${a_{12}, a_{13}, a_{14}, a_{23}, a_{24}, b_{12}}=0\,;$ & $X^{35/40}$ &
    $ a_{34}$ \\
  & ${a_{15}, a_{34}, b_{13}, c_{12}}\neq 0$ &  &  \\ \hline
6f. & ${a_{12}, a_{13}, a_{14}, a_{23}, b_{12}, b_{13}}=0\,;$ & $X^{34/40}$ &
     -  \\
  & ${a_{15}, a_{24}, b_{14}, b_{23}, c_{12}}\neq 0$ &  &  \\ \hline
6g. & ${a_{12}, a_{13}, a_{14}, a_{23}, b_{12}, c_{12}}=0\,;$ & $X^{34/40}$ &
     -  \\
  & ${a_{15}, a_{24}, b_{13}, d_{12}}\neq 0$ &  &  \\ \hline
6h. & ${a_{12}, a_{13}, a_{14}, b_{12}, b_{13}, b_{14}}=0\,;$ & $X^{34/40}$ &
     -  \\
  & ${a_{15}, a_{23}, c_{12}}\neq 0$ &  &  \\ \hline
6i. & ${a_{12}, a_{13}, a_{14}, b_{12}, b_{13}, c_{12}}=0\,;$ & $X^{34/40}$ &
     -  \\
  & ${a_{15}, a_{23}, b_{14}, c_{13}, d_{12}}\neq 0$ &  &  \\ \hline
6j. & ${a_{12}, a_{13}, a_{14}, b_{12}, c_{12}, d_{12}}=0\,;$ & $X^{34/40}$ &
     -  \\
  & ${a_{15}, a_{23}, b_{13}}\neq 0$ &  &  \\ \hline
6k. & ${a_{12}, a_{13}, a_{23}, b_{12}, b_{13}, b_{23}}=0\,;$ & $X^{34/40}$ &
     -  \\
  & ${a_{14}, c_{12}}\neq 0$ &  &  \\ \hline
6l. & ${a_{12}, a_{13}, a_{23}, b_{12}, b_{13}, c_{12}}=0\,;$ & $X^{34/40}$ &
     -  \\
  & ${a_{14}, b_{23}, c_{13}, d_{12}}\neq 0$ &  &  \\ \hline
6m. & ${a_{12}, a_{13}, a_{23}, b_{12}, c_{12}, d_{12}}=0\,;$ & $X^{34/40}$ &
     -  \\
  & ${a_{14}, b_{13}}\neq 0$ &  &  \\ \hline
\end{tabular}

\vspace{.15in}
\begin{center}
{\bf Table 1.} Subcases 5e--6m.
\end{center}

\begin{tabular}{|r|l|c|c|}\hline
Case & The set $S\subset V_\Z$ defined by & $N(S;X)\ll$ & Use factor \\ 
\hline\hline 
6n. & ${a_{12}, a_{13}, b_{12}, b_{13}, c_{12}, c_{13}}=0\,;$ & $X^{34/40}$ &
     -  \\
  & ${a_{14}, a_{23}, d_{12}}\neq 0$ &  &  \\ \hline
6o. & ${a_{12}, a_{13}, b_{12}, b_{13}, c_{12}, d_{12}}=0\,;$ & $X^{34/40}$ &
     -  \\
  & ${a_{14}, a_{23}, c_{13}}\neq 0$ &  &  \\ \hline
7a. & ${a_{12}, a_{13}, a_{14}, a_{15}, a_{23}, a_{24}, b_{12}}=0\,;$ & $X^{35
    /40}$ & $ a_{34}^2 $ \\
  & ${a_{25}, a_{34}, b_{13}, c_{12}}\neq 0$ &  &  \\ \hline
7b. & ${a_{12}, a_{13}, a_{14}, a_{15}, a_{23}, b_{12}, b_{13}}=0\,;$ & $X^{34
    /40}$ & $ a_{24}$ \\
  & ${a_{24}, b_{14}, b_{23}, c_{12}}\neq 0$ &  &  \\ \hline
7c. & ${a_{12}, a_{13}, a_{14}, a_{15}, a_{23}, b_{12}, c_{12}}=0\,;$ & $X^{34
    /40}$ & $ a_{24}$ \\
  & ${a_{24}, b_{13}, d_{12}}\neq 0$ &  &  \\ \hline
7d. & ${a_{12}, a_{13}, a_{14}, a_{15}, b_{12}, b_{13}, b_{14}}=0\,;$ & $X^{34
    /40}$ & $ b_{15}$ \\
  & ${a_{23}, b_{15}, c_{12}}\neq 0$ &  &  \\ \hline
7e. & ${a_{12}, a_{13}, a_{14}, a_{15}, b_{12}, b_{13}, c_{12}}=0\,;$ & $X^{34
    /40}$ & $ d_{12}$ \\
  & ${a_{23}, b_{14}, c_{13}, d_{12}}\neq 0$ &  &  \\ \hline
7f. & ${a_{12}, a_{13}, a_{14}, a_{15}, b_{12}, c_{12}, d_{12}}=0\,;$ & $X^{34
    /40}$ & $ b_{13}$ \\
  & ${a_{23}, b_{13}}\neq 0$ &  &  \\ \hline
7g. & ${a_{12}, a_{13}, a_{14}, a_{23}, a_{24}, b_{12}, b_{13}}=0\,;$ & $X^{34
    /40}$ & $ a_{34}$ \\
  & ${a_{15}, a_{34}, b_{14}, b_{23}, c_{12}}\neq 0$ &  &  \\ \hline
7h. & ${a_{12}, a_{13}, a_{14}, a_{23}, a_{24}, b_{12}, c_{12}}=0\,;$ & $X^{34
    /40}$ & $ a_{34}$ \\
  & ${a_{15}, a_{34}, b_{13}, d_{12}}\neq 0$ &  &  \\ \hline
7i. & ${a_{12}, a_{13}, a_{14}, a_{23}, b_{12}, b_{13}, b_{14}}=0\,;$ & $X^{33
    /40}$ &  -  \\
  & ${a_{15}, a_{24}, b_{23}, c_{12}}\neq 0$ &  &  \\ \hline
7j. & ${a_{12}, a_{13}, a_{14}, a_{23}, b_{12}, b_{13}, b_{23}}=0\,;$ & $X^{33
    /40}$ &  -  \\
  & ${a_{15}, a_{24}, b_{14}, c_{12}}\neq 0$ &  &  \\ \hline
7k. & ${a_{12}, a_{13}, a_{14}, a_{23}, b_{12}, b_{13}, c_{12}}=0\,;$ & $X^{33
    /40}$ &  -  \\
  & ${a_{15}, a_{24}, b_{14}, b_{23}, c_{13}, d_{12}}\neq 0$ &  &  \\ \hline
7l. & ${a_{12}, a_{13}, a_{14}, a_{23}, b_{12}, c_{12}, d_{12}}=0\,;$ & $X^{33
    /40}$ &  -  \\
  & ${a_{15}, a_{24}, b_{13}}\neq 0$ &  &  \\ \hline
7m. & ${a_{12}, a_{13}, a_{14}, b_{12}, b_{13}, b_{14}, c_{12}}=0\,;$ & $X^{34
    /40}$ & $ d_{12}$ \\
  & ${a_{15}, a_{23}, c_{13}, d_{12}}\neq 0$ &  &  \\ \hline
7n. & ${a_{12}, a_{13}, a_{14}, b_{12}, b_{13}, c_{12}, c_{13}}=0\,;$ & $X^{34
    /40}$ & $ d_{12}$ \\
  & ${a_{15}, a_{23}, b_{14}, d_{12}}\neq 0$ &  &  \\ \hline
7o. & ${a_{12}, a_{13}, a_{14}, b_{12}, b_{13}, c_{12}, d_{12}}=0\,;$ & $X^{34
    /40}$ & $ c_{13}$ \\
  & ${a_{15}, a_{23}, b_{14}, c_{13}}\neq 0$ &  &  \\ \hline
7p. & ${a_{12}, a_{13}, a_{23}, b_{12}, b_{13}, b_{23}, c_{12}}=0\,;$ & $X^{33
    /40}$ &  -  \\
  & ${a_{14}, c_{13}, d_{12}}\neq 0$ &  &  \\ \hline
7q. & ${a_{12}, a_{13}, a_{23}, b_{12}, b_{13}, c_{12}, c_{13}}=0\,;$ & $X^{33
    /40}$ &  -  \\
  & ${a_{14}, b_{23}, d_{12}}\neq 0$ &  &  \\ \hline
\end{tabular}

\vspace{.15in}
\begin{center}
{\bf Table 1.} Subcases 6n--7q.
\end{center}

\begin{tabular}{|r|l|c|c|}\hline
Case & The set $S\subset V_\Z$ defined by & $N(S;X)\ll$ & Use factor \\ 
\hline\hline 
7r. & ${a_{12}, a_{13}, a_{23}, b_{12}, b_{13}, c_{12}, d_{12}}=0\,;$ & $X^{33
    /40}$ &  -  \\
  & ${a_{14}, b_{23}, c_{13}}\neq 0$ &  &  \\ \hline
7s. & ${a_{12}, a_{13}, b_{12}, b_{13}, c_{12}, c_{13}, d_{12}}=0\,;$ & $X^{34
    /40}$ & $ d_{13}$ \\
  & ${a_{14}, a_{23}, d_{13}}\neq 0$ &  &  \\ \hline
8a. & ${a_{12}, a_{13}, a_{14}, a_{15}, a_{23}, a_{24}, b_{12}, b_{13}}
    =0\,;$ & $X^{34/40}$ & $ a_{34}^2 $ \\
  & ${a_{25}, a_{34}, b_{14}, b_{23}, c_{12}}\neq 0$ &  &  \\ \hline
8b. & ${a_{12}, a_{13}, a_{14}, a_{15}, a_{23}, a_{24}, b_{12}, c_{12}}
    =0\,;$ & $X^{34/40}$ & $ a_{25}a_{34}$ \\
  & ${a_{25}, a_{34}, b_{13}, d_{12}}\neq 0$ &  &  \\ \hline
8c. & ${a_{12}, a_{13}, a_{14}, a_{15}, a_{23}, b_{12}, b_{13}, b_{14}}
    =0\,;$ & $X^{34/40}$ & $ a_{24}b_{15}$ \\
  & ${a_{24}, b_{15}, b_{23}, c_{12}}\neq 0$ &  &  \\ \hline
8d. & ${a_{12}, a_{13}, a_{14}, a_{15}, a_{23}, b_{12}, b_{13}, b_{23}}
    =0\,;$ & $X^{33/40}$ & $ a_{24}$ \\
  & ${a_{24}, b_{14}, c_{12}}\neq 0$ &  &  \\ \hline
8e. & ${a_{12}, a_{13}, a_{14}, a_{15}, a_{23}, b_{12}, b_{13}, c_{12}}
    =0\,;$ & $X^{34/40}$ & $ a_{24}d_{12}$ \\
  & ${a_{24}, b_{14}, b_{23}, c_{13}, d_{12}}\neq 0$ &  &  \\ \hline
8f. & ${a_{12}, a_{13}, a_{14}, a_{15}, a_{23}, b_{12}, c_{12}, d_{12}}
    =0\,;$ & $X^{34/40}$ & $ a_{24}b_{13}$ \\
  & ${a_{24}, b_{13}}\neq 0$ &  &  \\ \hline
8g. & ${a_{12}, a_{13}, a_{14}, a_{15}, b_{12}, b_{13}, b_{14}, c_{12}}
    =0\,;$ & $X^{34/40}$ & $ b_{15}d_{12}$ \\
  & ${a_{23}, b_{15}, c_{13}, d_{12}}\neq 0$ &  &  \\ \hline
8h. & ${a_{12}, a_{13}, a_{14}, a_{15}, b_{12}, b_{13}, c_{12}, c_{13}}
    =0\,;$ & $X^{34/40}$ & $ b_{14}d_{12}$ \\
  & ${a_{23}, b_{14}, d_{12}}\neq 0$ &  &  \\ \hline
8i. & ${a_{12}, a_{13}, a_{14}, a_{15}, b_{12}, b_{13}, c_{12}, d_{12}}
    =0\,;$ & $X^{34/40}$ & $ c_{13}^2 $ \\
  & ${a_{23}, b_{14}, c_{13}}\neq 0$ &  &  \\ \hline
8j. & ${a_{12}, a_{13}, a_{14}, a_{23}, a_{24}, b_{12}, b_{13}, b_{14}}
    =0\,;$ & $X^{33/40}$ & $ a_{34}$ \\
  & ${a_{15}, a_{34}, b_{23}, c_{12}}\neq 0$ &  &  \\ \hline
8k. & ${a_{12}, a_{13}, a_{14}, a_{23}, a_{24}, b_{12}, b_{13}, b_{23}}
    =0\,;$ & $X^{33/40}$ & $ a_{34}$ \\
  & ${a_{15}, a_{34}, b_{14}, c_{12}}\neq 0$ &  &  \\ \hline
8l. & ${a_{12}, a_{13}, a_{14}, a_{23}, a_{24}, b_{12}, b_{13}, c_{12}}
    =0\,;$ & $X^{33/40}$ & $ a_{34}$ \\
  & ${a_{15}, a_{34}, b_{14}, b_{23}, c_{13}, d_{12}}\neq 0$ &  &  \\ \hline
8m. & ${a_{12}, a_{13}, a_{14}, a_{23}, a_{24}, b_{12}, c_{12}, d_{12}}
    =0\,;$ & $X^{33/40}$ & $ a_{15}$ \\
  & ${a_{15}, a_{34}, b_{13}}\neq 0$ &  &  \\ \hline
8n. & ${a_{12}, a_{13}, a_{14}, a_{23}, b_{12}, b_{13}, b_{14}, b_{23}}
    =0\,;$ & $X^{33/40}$ & $ a_{24}$ \\
  & ${a_{15}, a_{24}, c_{12}}\neq 0$ &  &  \\ \hline
8o. & ${a_{12}, a_{13}, a_{14}, a_{23}, b_{12}, b_{13}, b_{14}, c_{12}}
    =0\,;$ & $X^{32/40}$ &  -  \\
  & ${a_{15}, a_{24}, b_{23}, c_{13}, d_{12}}\neq 0$ &  &  \\ \hline
8p. & ${a_{12}, a_{13}, a_{14}, a_{23}, b_{12}, b_{13}, b_{23}, c_{12}}
    =0\,;$ & $X^{32/40}$ &  -  \\
  & ${a_{15}, a_{24}, b_{14}, c_{13}, d_{12}}\neq 0$ &  &  \\ \hline
8q. & ${a_{12}, a_{13}, a_{14}, a_{23}, b_{12}, b_{13}, c_{12}, c_{13}}
    =0\,;$ & $X^{32/40}$ &  -  \\
  & ${a_{15}, a_{24}, b_{14}, b_{23}, d_{12}}\neq 0$ &  &  \\ \hline
\end{tabular}

\vspace{.15in}
\begin{center}
{\bf Table 1.} Subcases 7r--8q.
\end{center}

\begin{tabular}{|r|l|c|c|}\hline
Case & The set $S\subset V_\Z$ defined by & $N(S;X)\ll$ & Use factor \\ 
\hline\hline 
8r. & ${a_{12}, a_{13}, a_{14}, a_{23}, b_{12}, b_{13}, c_{12}, d_{12}}
    =0\,;$ & $X^{32/40}$ &  -  \\
  & ${a_{15}, a_{24}, b_{14}, b_{23}, c_{13}}\neq 0$ &  &  \\ \hline
8s. & ${a_{12}, a_{13}, a_{14}, b_{12}, b_{13}, b_{14}, c_{12}, c_{13}}
    =0\,;$ & $X^{34/40}$ & $ c_{14}d_{12}$ \\
  & ${a_{15}, a_{23}, c_{14}, d_{12}}\neq 0$ &  &  \\ \hline
8t. & ${a_{12}, a_{13}, a_{14}, b_{12}, b_{13}, b_{14}, c_{12}, d_{12}}
    =0\,;$ & $X^{34/40}$ & $ c_{13}^2 $ \\
  & ${a_{15}, a_{23}, c_{13}}\neq 0$ &  &  \\ \hline
8u. & ${a_{12}, a_{13}, a_{14}, b_{12}, b_{13}, c_{12}, c_{13}, d_{12}}
    =0\,;$ & $X^{34/40}$ & $ d_{13}^2 $ \\
  & ${a_{15}, a_{23}, b_{14}, d_{13}}\neq 0$ &  &  \\ \hline
8v. & ${a_{12}, a_{13}, a_{23}, b_{12}, b_{13}, b_{23}, c_{12}, c_{13}}
    =0\,;$ & $X^{33/40}$ & $ d_{12}$ \\
  & ${a_{14}, c_{23}, d_{12}}\neq 0$ &  &  \\ \hline
8w. & ${a_{12}, a_{13}, a_{23}, b_{12}, b_{13}, b_{23}, c_{12}, d_{12}}
    =0\,;$ & $X^{33/40}$ & $ c_{13}$ \\
  & ${a_{14}, c_{13}}\neq 0$ &  &  \\ \hline
8x. & ${a_{12}, a_{13}, a_{23}, b_{12}, b_{13}, c_{12}, c_{13}, d_{12}}
    =0\,;$ & $X^{33/40}$ & $ d_{13}$ \\
  & ${a_{14}, b_{23}, d_{13}}\neq 0$ &  &  \\ \hline
9a. & ${a_{12}, a_{13}, a_{14}, a_{15}, a_{23}, a_{24}, b_{12}, b_{13},
     b_{14}}=0\,;$ & $X^{34/40}$ & $ a_{34}^2 b_{15}$ \\
  & ${a_{25}, a_{34}, b_{15}, b_{23}, c_{12}}\neq 0$ &  &  \\ \hline
9b. & ${a_{12}, a_{13}, a_{14}, a_{15}, a_{23}, a_{24}, b_{12}, b_{13},
     b_{23}}=0\,;$ & $X^{33/40}$ & $ a_{34}^2 $ \\
  & ${a_{25}, a_{34}, b_{14}, c_{12}}\neq 0$ &  &  \\ \hline
9c. & ${a_{12}, a_{13}, a_{14}, a_{15}, a_{23}, a_{24}, b_{12}, b_{13},
     c_{12}}=0\,;$ & $X^{34/40}$ & $ a_{34}^2 d_{12}$ \\
  & ${a_{25}, a_{34}, b_{14}, b_{23}, c_{13}, d_{12}}\neq 0$ &  &  \\ \hline
9d. & ${a_{12}, a_{13}, a_{14}, a_{15}, a_{23}, a_{24}, b_{12}, c_{12},
     d_{12}}=0\,;$ & $X^{34/40}$ & $ a_{25}^2 b_{13}$ \\
  & ${a_{25}, a_{34}, b_{13}}\neq 0$ &  &  \\ \hline
9e. & ${a_{12}, a_{13}, a_{14}, a_{15}, a_{23}, b_{12}, b_{13}, b_{14},
     b_{23}}=0\,;$ & $X^{33/40}$ & $ a_{24}b_{15}$ \\
  & ${a_{24}, b_{15}, c_{12}}\neq 0$ &  &  \\ \hline
9f. & ${a_{12}, a_{13}, a_{14}, a_{15}, a_{23}, b_{12}, b_{13}, b_{14},
     c_{12}}=0\,;$ & $X^{34/40}$ & $ a_{24}b_{15}d_{12}$ \\
  & ${a_{24}, b_{15}, b_{23}, c_{13}, d_{12}}\neq 0$ &  &  \\ \hline
9g. & ${a_{12}, a_{13}, a_{14}, a_{15}, a_{23}, b_{12}, b_{13}, b_{23},
     c_{12}}=0\,;$ & $X^{31/40}$ &  -  \\
  & ${a_{24}, b_{14}, c_{13}, d_{12}}\neq 0$ &  &  \\ \hline
9h. & ${a_{12}, a_{13}, a_{14}, a_{15}, a_{23}, b_{12}, b_{13}, c_{12},
     c_{13}}=0\,;$ & $X^{34/40}$ & $ a_{24}b_{14}d_{12}$ \\
  & ${a_{24}, b_{14}, b_{23}, d_{12}}\neq 0$ &  &  \\ \hline
9i. & ${a_{12}, a_{13}, a_{14}, a_{15}, a_{23}, b_{12}, b_{13}, c_{12},
     d_{12}}=0\,;$ & $X^{34/40}$ & $ a_{24}c_{13}^2 $ \\
  & ${a_{24}, b_{14}, b_{23}, c_{13}}\neq 0$ &  &  \\ \hline
9j. & ${a_{12}, a_{13}, a_{14}, a_{15}, b_{12}, b_{13}, b_{14}, c_{12},
     c_{13}}=0\,;$ & $X^{34/40}$ & $ b_{15}c_{14}d_{12}$ \\
  & ${a_{23}, b_{15}, c_{14}, d_{12}}\neq 0$ &  &  \\ \hline
9k. & ${a_{12}, a_{13}, a_{14}, a_{15}, b_{12}, b_{13}, b_{14}, c_{12},
     d_{12}}=0\,;$ & $X^{34/40}$ & $ b_{15}c_{13}^2 $ \\
  & ${a_{23}, b_{15}, c_{13}}\neq 0$ &  &  \\ \hline
9l. & ${a_{12}, a_{13}, a_{14}, a_{15}, b_{12}, b_{13}, c_{12}, c_{13},
     d_{12}}=0\,;$ & $X^{34/40}$ & $ b_{14}d_{13}^2 $ \\
  & ${a_{23}, b_{14}, d_{13}}\neq 0$ &  &  \\ \hline
\end{tabular}

\vspace{.15in}
\begin{center}
{\bf Table 1.} Subcases 8r--9l.
\end{center}

\begin{tabular}{|r|l|c|c|}\hline
Case & The set $S\subset V_\Z$ defined by & $N(S;X)\ll$ & Use factor \\ 
\hline\hline 
9m. & ${a_{12}, a_{13}, a_{14}, a_{23}, a_{24}, b_{12}, b_{13}, b_{14},
     b_{23}}=0\,;$ & $X^{33/40}$ & $ a_{34}b_{24}$ \\
  & ${a_{15}, a_{34}, b_{24}, c_{12}}\neq 0$ &  &  \\ \hline
9n. & ${a_{12}, a_{13}, a_{14}, a_{23}, a_{24}, b_{12}, b_{13}, b_{14},
     c_{12}}=0\,;$ & $X^{31/40}$ &  -  \\
  & ${a_{15}, a_{34}, b_{23}, c_{13}, d_{12}}\neq 0$ &  &  \\ \hline
9o. & ${a_{12}, a_{13}, a_{14}, a_{23}, a_{24}, b_{12}, b_{13}, b_{23},
     c_{12}}=0\,;$ & $X^{31/40}$ &  -  \\
  & ${a_{15}, a_{34}, b_{14}, c_{13}, d_{12}}\neq 0$ &  &  \\ \hline
9p. & ${a_{12}, a_{13}, a_{14}, a_{23}, a_{24}, b_{12}, b_{13}, c_{12},
     c_{13}}=0\,;$ & $X^{31/40}$ &  -  \\
  & ${a_{15}, a_{34}, b_{14}, b_{23}, d_{12}}\neq 0$ &  &  \\ \hline
9q. & ${a_{12}, a_{13}, a_{14}, a_{23}, a_{24}, b_{12}, b_{13}, c_{12},
     d_{12}}=0\,;$ & $X^{31/40}$ &  -  \\
  & ${a_{15}, a_{34}, b_{14}, b_{23}, c_{13}}\neq 0$ &  &  \\ \hline
9r. & ${a_{12}, a_{13}, a_{14}, a_{23}, b_{12}, b_{13}, b_{14}, b_{23},
     c_{12}}=0\,;$ & $X^{31/40}$ &  -  \\
  & ${a_{15}, a_{24}, c_{13}, d_{12}}\neq 0$ &  &  \\ \hline
9s. & ${a_{12}, a_{13}, a_{14}, a_{23}, b_{12}, b_{13}, b_{14}, c_{12},
     c_{13}}=0\,;$ & $X^{32/40}$ & $ c_{14}$ \\
  & ${a_{15}, a_{24}, b_{23}, c_{14}, d_{12}}\neq 0$ &  &  \\ \hline
9t. & ${a_{12}, a_{13}, a_{14}, a_{23}, b_{12}, b_{13}, b_{14}, c_{12},
     d_{12}}=0\,;$ & $X^{32/40}$ & $ c_{13}$ \\
  & ${a_{15}, a_{24}, b_{23}, c_{13}}\neq 0$ &  &  \\ \hline
9u. & ${a_{12}, a_{13}, a_{14}, a_{23}, b_{12}, b_{13}, b_{23}, c_{12},
     c_{13}}=0\,;$ & $X^{32/40}$ & $ c_{23}$ \\
  & ${a_{15}, a_{24}, b_{14}, c_{23}, d_{12}}\neq 0$ &  &  \\ \hline
9v. & ${a_{12}, a_{13}, a_{14}, a_{23}, b_{12}, b_{13}, b_{23}, c_{12},
     d_{12}}=0\,;$ & $X^{32/40}$ & $ c_{13}$ \\
  & ${a_{15}, a_{24}, b_{14}, c_{13}}\neq 0$ &  &  \\ \hline
9w. & ${a_{12}, a_{13}, a_{14}, a_{23}, b_{12}, b_{13}, c_{12}, c_{13},
     d_{12}}=0\,;$ & $X^{32/40}$ & $ d_{13}$ \\
  & ${a_{15}, a_{24}, b_{14}, b_{23}, d_{13}}\neq 0$ &  &  \\ \hline
9x. & ${a_{12}, a_{13}, a_{14}, b_{12}, b_{13}, b_{14}, c_{12}, c_{13},
     d_{12}}=0\,;$ & $X^{34/40}$ & $ c_{14}d_{13}^2 $ \\
  & ${a_{15}, a_{23}, c_{14}, d_{13}}\neq 0$ &  &  \\ \hline
9y. & ${a_{12}, a_{13}, a_{23}, b_{12}, b_{13}, b_{23}, c_{12}, c_{13},
     d_{12}}=0\,;$ & $X^{33/40}$ & $ d_{13}^2 $ \\
  & ${a_{14}, c_{23}, d_{13}}\neq 0$ &  &  \\ \hline
10a. & ${a_{12}, a_{13}, a_{14}, a_{15}, a_{23}, a_{24}, b_{12}, b_{13},
     b_{14}, b_{23}}=0\,;$ & $X^{33/40}$ & $ a_{34}^2 b_{15}$ \\
  & ${a_{25}, a_{34}, b_{15}, b_{24}, c_{12}}\neq 0$ &  &  \\ \hline
10b. & ${a_{12}, a_{13}, a_{14}, a_{15}, a_{23}, a_{24}, b_{12}, b_{13},
     b_{14}, c_{12}}=0\,;$ & $X^{34/40}$ & $ a_{34}^2 b_{15}d_{12}$ \\
  & ${a_{25}, a_{34}, b_{15}, b_{23}, c_{13}, d_{12}}\neq 0$ &  &  \\ \hline
10c. & ${a_{12}, a_{13}, a_{14}, a_{15}, a_{23}, a_{24}, b_{12}, b_{13},
     b_{23}, c_{12}}=0\,;$ & $X^{31/40}$ & $ a_{34}$ \\
  & ${a_{25}, a_{34}, b_{14}, c_{13}, d_{12}}\neq 0$ &  &  \\ \hline
10d. & ${a_{12}, a_{13}, a_{14}, a_{15}, a_{23}, a_{24}, b_{12}, b_{13},
     c_{12}, c_{13}}=0\,;$ & $X^{34/40}$ & $ a_{34}^2 b_{14}d_{12}$ \\
  & ${a_{25}, a_{34}, b_{14}, b_{23}, d_{12}}\neq 0$ &  &  \\ \hline
10e. & ${a_{12}, a_{13}, a_{14}, a_{15}, a_{23}, a_{24}, b_{12}, b_{13},
     c_{12}, d_{12}}=0\,;$ & $X^{34/40}$ & $ a_{25}^2 c_{13}^2 $ \\
  & ${a_{25}, a_{34}, b_{14}, b_{23}, c_{13}}\neq 0$ &  &  \\ \hline
10f. & ${a_{12}, a_{13}, a_{14}, a_{15}, a_{23}, b_{12}, b_{13}, b_{14},
     b_{23}, c_{12}}=0\,;$ & $X^{31/40}$ & $ b_{15}$ \\
  & ${a_{24}, b_{15}, c_{13}, d_{12}}\neq 0$ &  &  \\ \hline
\end{tabular}

\vspace{.15in}
\begin{center}
{\bf Table 1.} Subcases 9m--10f.
\end{center}

\begin{tabular}{|r|l|c|c|}\hline
Case & The set $S\subset V_\Z$ defined by & $N(S;X)\ll$ & Use factor \\ 
\hline\hline 
10g. & ${a_{12}, a_{13}, a_{14}, a_{15}, a_{23}, b_{12}, b_{13}, b_{14},
     c_{12}, c_{13}}=0\,;$ & $X^{34/40}$ & $ a_{24}b_{15}c_{14}d_{12}$ \\
  & ${a_{24}, b_{15}, b_{23}, c_{14}, d_{12}}\neq 0$ &  &  \\ \hline
10h. & ${a_{12}, a_{13}, a_{14}, a_{15}, a_{23}, b_{12}, b_{13}, b_{14},
     c_{12}, d_{12}}=0\,;$ & $X^{34/40}$ & $ a_{24}b_{15}c_{13}^2 $ \\
  & ${a_{24}, b_{15}, b_{23}, c_{13}}\neq 0$ &  &  \\ \hline
10i. & ${a_{12}, a_{13}, a_{14}, a_{15}, a_{23}, b_{12}, b_{13}, b_{23},
     c_{12}, c_{13}}=0\,;$ & $X^{33/40}$ & $ a_{24}b_{14}d_{12}$ \\
  & ${a_{24}, b_{14}, c_{23}, d_{12}}\neq 0$ &  &  \\ \hline
10j. & ${a_{12}, a_{13}, a_{14}, a_{15}, a_{23}, b_{12}, b_{13}, b_{23},
     c_{12}, d_{12}}=0\,;$ & $X^{31/40}$ & $ c_{13}$ \\
  & ${a_{24}, b_{14}, c_{13}}\neq 0$ &  &  \\ \hline
10k. & ${a_{12}, a_{13}, a_{14}, a_{15}, a_{23}, b_{12}, b_{13}, c_{12},
     c_{13}, d_{12}}=0\,;$ & $X^{34/40}$ & $ a_{24}b_{14}d_{13}^2 $ \\
  & ${a_{24}, b_{14}, b_{23}, d_{13}}\neq 0$ &  &  \\ \hline
10l. & ${a_{12}, a_{13}, a_{14}, a_{15}, b_{12}, b_{13}, b_{14}, c_{12},
     c_{13}, d_{12}}=0\,;$ & $X^{34/40}$ & $ b_{15}c_{14}d_{13}^2 $ \\
  & ${a_{23}, b_{15}, c_{14}, d_{13}}\neq 0$ &  &  \\ \hline
10m. & ${a_{12}, a_{13}, a_{14}, a_{23}, a_{24}, b_{12}, b_{13}, b_{14},
     b_{23}, c_{12}}=0\,;$ & $X^{31/40}$ & $ b_{24}$ \\
  & ${a_{15}, a_{34}, b_{24}, c_{13}, d_{12}}\neq 0$ &  &  \\ \hline
10n. & ${a_{12}, a_{13}, a_{14}, a_{23}, a_{24}, b_{12}, b_{13}, b_{14},
     c_{12}, c_{13}}=0\,;$ & $X^{31/40}$ & $ c_{14}$ \\
  & ${a_{15}, a_{34}, b_{23}, c_{14}, d_{12}}\neq 0$ &  &  \\ \hline
10o. & ${a_{12}, a_{13}, a_{14}, a_{23}, a_{24}, b_{12}, b_{13}, b_{14},
     c_{12}, d_{12}}=0\,;$ & $X^{31/40}$ & $ c_{13}$ \\
  & ${a_{15}, a_{34}, b_{23}, c_{13}}\neq 0$ &  &  \\ \hline
10p. & ${a_{12}, a_{13}, a_{14}, a_{23}, a_{24}, b_{12}, b_{13}, b_{23},
     c_{12}, c_{13}}=0\,;$ & $X^{31/40}$ & $ c_{23}$ \\
  & ${a_{15}, a_{34}, b_{14}, c_{23}, d_{12}}\neq 0$ &  &  \\ \hline
10q. & ${a_{12}, a_{13}, a_{14}, a_{23}, a_{24}, b_{12}, b_{13}, b_{23},
     c_{12}, d_{12}}=0\,;$ & $X^{31/40}$ & $ c_{13}$ \\
  & ${a_{15}, a_{34}, b_{14}, c_{13}}\neq 0$ &  &  \\ \hline
10r. & ${a_{12}, a_{13}, a_{14}, a_{23}, a_{24}, b_{12}, b_{13}, c_{12},
     c_{13}, d_{12}}=0\,;$ & $X^{31/40}$ & $ d_{13}$ \\
  & ${a_{15}, a_{34}, b_{14}, b_{23}, d_{13}}\neq 0$ &  &  \\ \hline
10s. & ${a_{12}, a_{13}, a_{14}, a_{23}, b_{12}, b_{13}, b_{14}, b_{23},
     c_{12}, c_{13}}=0\,;$ & $X^{33/40}$ & $ a_{24}c_{14}d_{12}$ \\
  & ${a_{15}, a_{24}, c_{14}, c_{23}, d_{12}}\neq 0$ &  &  \\ \hline
10t. & ${a_{12}, a_{13}, a_{14}, a_{23}, b_{12}, b_{13}, b_{14}, b_{23},
     c_{12}, d_{12}}=0\,;$ & $X^{31/40}$ & $ c_{13}$ \\
  & ${a_{15}, a_{24}, c_{13}}\neq 0$ &  &  \\ \hline
10u. & ${a_{12}, a_{13}, a_{14}, a_{23}, b_{12}, b_{13}, b_{14}, c_{12},
     c_{13}, d_{12}}=0\,;$ & $X^{32/40}$ & $ c_{14}d_{13}$ \\
  & ${a_{15}, a_{24}, b_{23}, c_{14}, d_{13}}\neq 0$ &  &  \\ \hline
10v. & ${a_{12}, a_{13}, a_{14}, a_{23}, b_{12}, b_{13}, b_{23}, c_{12},
     c_{13}, d_{12}}=0\,;$ & $X^{32/40}$ & $ c_{23}d_{13}$ \\
  & ${a_{15}, a_{24}, b_{14}, c_{23}, d_{13}}\neq 0$ &  &  \\ \hline
11a. & ${a_{12}, a_{13}, a_{14}, a_{15}, a_{23}, a_{24}, b_{12}, b_{13},
     b_{14}, b_{23}, c_{12}}=0\,;$ & $X^{31/40}$ & $ a_{34}b_{15}$ \\
  & ${a_{25}, a_{34}, b_{15}, b_{24}, c_{13}, d_{12}}\neq 0$ &  &  \\ \hline
11b. & ${a_{12}, a_{13}, a_{14}, a_{15}, a_{23}, a_{24}, b_{12}, b_{13},
     b_{14}, c_{12}, c_{13}}=0\,;$ & $X^{34/40}$ &
    $ a_{34}^2 b_{15}c_{14}d_{12}$ \\
  & ${a_{25}, a_{34}, b_{15}, b_{23}, c_{14}, d_{12}}\neq 0$ &  &  \\ \hline
11c. & ${a_{12}, a_{13}, a_{14}, a_{15}, a_{23}, a_{24}, b_{12}, b_{13},
     b_{14}, c_{12}, d_{12}}=0\,;$ & $X^{36/40}$ &
    $ a_{25}^2 a_{34}b_{15}c_{13}^3 $ \\
  & ${a_{25}, a_{34}, b_{15}, b_{23}, c_{13}}\neq 0$ &  &  \\ \hline
\end{tabular}

\vspace{.15in}
\begin{center}
{\bf Table 1.} Subcases 10g--11c.
\end{center}

\begin{tabular}{|r|l|c|c|}\hline
Case & The set $S\subset V_\Z$ defined by & $N(S;X)\ll$ & Use factor \\ 
\hline\hline 
11d. & ${a_{12}, a_{13}, a_{14}, a_{15}, a_{23}, a_{24}, b_{12}, b_{13},
     b_{23}, c_{12}, c_{13}}=0\,;$ & $X^{33/40}$ & $ a_{34}^2 b_{14}d_{12}$
     \\
  & ${a_{25}, a_{34}, b_{14}, c_{23}, d_{12}}\neq 0$ &  &  \\ \hline
11e. & ${a_{12}, a_{13}, a_{14}, a_{15}, a_{23}, a_{24}, b_{12}, b_{13},
     b_{23}, c_{12}, d_{12}}=0\,;$ & $X^{31/40}$ & $ a_{25}c_{13}$ \\
  & ${a_{25}, a_{34}, b_{14}, c_{13}}\neq 0$ &  &  \\ \hline
11f. & ${a_{12}, a_{13}, a_{14}, a_{15}, a_{23}, a_{24}, b_{12}, b_{13},
     c_{12}, c_{13}, d_{12}}=0\,;$ & $X^{34/40}$ & $ a_{25}^2 b_{14}d_{13}^2 $
     \\
  & ${a_{25}, a_{34}, b_{14}, b_{23}, d_{13}}\neq 0$ &  &  \\ \hline
11g. & ${a_{12}, a_{13}, a_{14}, a_{15}, a_{23}, b_{12}, b_{13}, b_{14},
     b_{23}, c_{12}, c_{13}}=0\,;$ & $X^{33/40}$ & $ a_{24}b_{15}c_{14}d_{12}$
     \\
  & ${a_{24}, b_{15}, c_{14}, c_{23}, d_{12}}\neq 0$ &  &  \\ \hline
11h. & ${a_{12}, a_{13}, a_{14}, a_{15}, a_{23}, b_{12}, b_{13}, b_{14},
     b_{23}, c_{12}, d_{12}}=0\,;$ & $X^{31/40}$ & $ b_{15}c_{13}$ \\
  & ${a_{24}, b_{15}, c_{13}}\neq 0$ &  &  \\ \hline
11i. & ${a_{12}, a_{13}, a_{14}, a_{15}, a_{23}, b_{12}, b_{13}, b_{14},
     c_{12}, c_{13}, d_{12}}=0\,;$ & $X^{34/40}$ &
    $ a_{24}b_{15}c_{14}d_{13}^2 $ \\
  & ${a_{24}, b_{15}, b_{23}, c_{14}, d_{13}}\neq 0$ &  &  \\ \hline
11j. & ${a_{12}, a_{13}, a_{14}, a_{15}, a_{23}, b_{12}, b_{13}, b_{23},
     c_{12}, c_{13}, d_{12}}=0\,;$ & $X^{33/40}$ & $ a_{24}b_{14}d_{13}^2 $
     \\
  & ${a_{24}, b_{14}, c_{23}, d_{13}}\neq 0$ &  &  \\ \hline
11k. & ${a_{12}, a_{13}, a_{14}, a_{23}, a_{24}, b_{12}, b_{13}, b_{14},
     b_{23}, c_{12}, c_{13}}=0\,;$ & $X^{33/40}$ & $ a_{34}b_{24}c_{14}d_{12}$
     \\
  & ${a_{15}, a_{34}, b_{24}, c_{14}, c_{23}, d_{12}}\neq 0$ &  &  \\ \hline
11l. & ${a_{12}, a_{13}, a_{14}, a_{23}, a_{24}, b_{12}, b_{13}, b_{14},
     b_{23}, c_{12}, d_{12}}=0\,;$ & $X^{31/40}$ & $ b_{24}c_{13}$ \\
  & ${a_{15}, a_{34}, b_{24}, c_{13}}\neq 0$ &  &  \\ \hline
11m. & ${a_{12}, a_{13}, a_{14}, a_{23}, a_{24}, b_{12}, b_{13}, b_{14},
     c_{12}, c_{13}, d_{12}}=0\,;$ & $X^{31/40}$ & $ c_{14}d_{13}$ \\
  & ${a_{15}, a_{34}, b_{23}, c_{14}, d_{13}}\neq 0$ &  &  \\ \hline
11n. & ${a_{12}, a_{13}, a_{14}, a_{23}, a_{24}, b_{12}, b_{13}, b_{23},
     c_{12}, c_{13}, d_{12}}=0\,;$ & $X^{31/40}$ & $ c_{23}d_{13}$ \\
  & ${a_{15}, a_{34}, b_{14}, c_{23}, d_{13}}\neq 0$ &  &  \\ \hline
11o. & ${a_{12}, a_{13}, a_{14}, a_{23}, b_{12}, b_{13}, b_{14}, b_{23},
     c_{12}, c_{13}, d_{12}}=0\,;$ & $X^{33/40}$ & $ a_{24}c_{14}d_{13}^2 $
     \\
  & ${a_{15}, a_{24}, c_{14}, c_{23}, d_{13}}\neq 0$ &  &  \\ \hline
12a. & ${a_{12}, a_{13}, a_{14}, a_{15}, a_{23}, a_{24}, b_{12}, b_{13},
     b_{14}, b_{23}, c_{12}, c_{13}}=0\,;$ & $X^{33/40}$ &
    $ a_{34}^2 b_{15}c_{14}d_{12}$ \\
  & ${a_{25}, a_{34}, b_{15}, b_{24}, c_{14}, c_{23}, d_{12}}
    \neq 0$ &  &  \\ \hline
12b. & ${a_{12}, a_{13}, a_{14}, a_{15}, a_{23}, a_{24}, b_{12}, b_{13},
     b_{14}, b_{23}, c_{12}, d_{12}}=0\,;$ & $X^{36/40}$ &
    $ a_{25}^2 a_{34}b_{15}b_{24}c_{13}^3 $ \\
  & ${a_{25}, a_{34}, b_{15}, b_{24}, c_{13}}\neq 0$ &  &  \\ \hline
12c. & ${a_{12}, a_{13}, a_{14}, a_{15}, a_{23}, a_{24}, b_{12}, b_{13},
     b_{14}, c_{12}, c_{13}, d_{12}}=0\,;$ & $X^{36/40}$ &
    $ a_{25}^2 a_{34}b_{15}c_{14}^2 d_{13}^2 $ \\
  & ${a_{25}, a_{34}, b_{15}, b_{23}, c_{14}, d_{13}}\neq 0$ &  &  \\ \hline
12d. & ${a_{12}, a_{13}, a_{14}, a_{15}, a_{23}, a_{24}, b_{12}, b_{13},
     b_{23}, c_{12}, c_{13}, d_{12}}=0\,;$ & $X^{33/40}$ &
    $ a_{25}^2 b_{14}d_{13}^2 $ \\
  & ${a_{25}, a_{34}, b_{14}, c_{23}, d_{13}}\neq 0$ &  &  \\ \hline
12e. & ${a_{12}, a_{13}, a_{14}, a_{15}, a_{23}, b_{12}, b_{13}, b_{14},
     b_{23}, c_{12}, c_{13}, d_{12}}=0\,;$ & $X^{33/40}$ &
    $ a_{24}b_{15}c_{14}d_{13}^2 $ \\
  & ${a_{24}, b_{15}, c_{14}, c_{23}, d_{13}}\neq 0$ &  &  \\ \hline
12f. & ${a_{12}, a_{13}, a_{14}, a_{23}, a_{24}, b_{12}, b_{13}, b_{14},
     b_{23}, c_{12}, c_{13}, d_{12}}=0\,;$ & $X^{33/40}$ &
    $ a_{15}b_{24}c_{23}d_{13}^2 $ \\
  & ${a_{15}, a_{34}, b_{24}, c_{14}, c_{23}, d_{13}}\neq 0$ &  &  \\ \hline
13. & ${a_{12}, a_{13}, a_{14}, a_{15}, a_{23}, a_{24}, b_{12}, b_{13},
     b_{14}, b_{23}, c_{12}, c_{13}, d_{12}}=0\,;$ & $X^{37/40}$ &
     $ a_{25}^2 a_{34}b_{24}^2 c_{14}^2 d_{13}^3 $ \\
  & ${a_{25}, a_{34}, b_{15}, b_{24}, c_{14}, c_{23}, d_{13}}
     \neq 0$ &  &  \\ \hline
\end{tabular}
 
\vspace{.15in}
\begin{center}
{\bf Table 1.} Subcases 11d--13.
\end{center} 

\vspace{.35in}


\pagebreak
Therefore, for the purposes of proving Theorem~\ref{cna}, we may assume 
that $a_{12}\neq 0$.  

\subsection{The main term}\label{mainterm}

Let $\RR_X(v)$ denote the multiset $\{x\in
\FF v:|\Disc(x)|<X\}$.  Then we have the following result counting the
number of integral points in $\RR_X(v)$, on average, satisfying
$a_{12}\neq 0$:

\begin{proposition}\label{nonzeroa12}
  Let $v$ take a random value in $H\cap V^{(i)}$ uniformly with
  respect to the measure $|\Disc(v)|^{-1}\,dv$.  Then the expected
  number of integral elements $(A,B,C,D)\in\FF v$ such that \linebreak
  $|\Disc(A,B,C,D)|< X$ and $a_{12}\neq0$ is $\Vol(\RR_X(v_i))
  + O(X^{39/40})$, where $v_i$ is any vector in $V^{(i)}$.
\end{proposition}

\begin{proof}
Following the proof of Lemma~\ref{hard}, let $V^{(i)}(0)$ denote the subset
of $V_\R$ such that $a_{12}\neq 0$.  
We wish to show that
\begin{equation}\label{toprove2}
N^*(V^{(i)}(0);X)=\frac{1}{n_i}\cdot\Vol(\RR_X(v_i)) + O(X^{39/40}).
\end{equation}
We have
\begin{equation}\label{translate2}
  N^*(V^{(i)}(0);X)=\frac{r_i}{M_i}\int_{\lambda=c'}^{X^{1/40}} \!\!
\int_{s_1,s_2,\ldots,s_7=c}^\infty
\int_{u\in    N'(a(s))} 
  \sigma(V(0))
  \, s_1^{-12}s_2^{-8}s_3^{-12}s_4^{-20}s_5^{-30}s_6^{-30}s_7^{-20} \,du\,
  d^\times\! s \,d^\times\!\lambda,
\end{equation}
where $\sigma(V(0))$ denotes the number of integer points
in the region $H(u,s,\lambda,X)$ satisfying $|a_{12}|\geq 1$.
Evidently, the number of integer points in $H(u,s,\lambda,X)$
with $|a_{12}|\geq 1$ can be nonzero only if we have
\begin{equation}\label{condt2}
J{w(a_{12})}=J\cdot\frac{\lambda}{s_1^3s_2s_3s_4^3s_5^6s_6^4s_7^2}\geq 1.
\end{equation}
Therefore, if the region $\H=\{(A,B,C,D)\in
H(u,s,\lambda,X):|a_{12}|\geq 1\}$
 contains an integer point, then (\ref{condt2}) and
Lemma~\ref{genbound} imply that the number of integer points in $\H$
is $\Vol(\H)+O(J^{-1}\Vol(\H)/w(a_{12}))$, since all
smaller-dimensional projections of $u^{-1}\H$ are clearly bounded by
a constant times the projection of $\H$ onto the hyperplane $a_{12}=0$ (since
$a_{12}$ has minimal weight).

Therefore, since $\H=H(u,s,\lambda,X)-\bigl(H(u,s,\lambda,X)-\H\bigr)$, 
we may write
\begin{eqnarray}\label{bigint}
N^\ast(V^{(i)}(0);X) &\!\!\!=\!\!& \!\frac{r_i}{M_i}
\int_{\lambda=c'}^{X^{1/40}}
\int_{s_1,\ldots,s_7=c}^{\infty} \int_{u\in N'(a(s))}
\Bigl(\Vol\bigl(H(u,s,\lambda,X)\bigr)-\Vol\bigl(H(u,s,\lambda,X)-\H\bigr) 
+ \\[.085in]\nonumber & & \,\,\,\,\,\,\,\,\,\,
\,\,\,\,\,\,\,O(\max\{{J^{39}\lambda^{39}s_1^3s_2s_3s_4^3s_5^6s_6^4s_7^2},1\})
\Bigr)   
\, s_1^{-12}s_2^{-8}s_3^{-12}s_4^{-20}s_5^{-30}s_6^{-30}s_7^{-20} \,du\,
d^\times s\, d^\times \lambda.
\end{eqnarray}
The integral of the first term in (\ref{bigint}) is $(1/r_i)\cdot\int_{v\in H\cap V^{(i)}}
\Vol(\RR_X(v))|\Disc(v)|^{-1}dv$.
Since $\Vol(\RR_X(v))$ does not
depend on the choice of $v\in V^{(i)}$ (see Section~\ref{volcomp}),
the latter integral is simply $[M_i/(n_i\,r_i)]\cdot \Vol(\RR_X(v))$.

To estimate the integral of the second term in (\ref{bigint}), let $\H'=
H(u,s,t,X)-\H$, and for each $|a_{12}|\leq 1$, let $\H'(a_{12})$ be
the subset of all elements $(A,B,C,D)\in\H'$ with the given value of
$a_{12}$.  Then the 39-dimensional volume of $\H'(a_{12})$ is at most 
$O\Bigl(J^{39}\prod_{t\in T\setminus\{a_{12}\}}w(t)\Bigr)$, and so we
have the estimate 
\[\Vol(\H') \ll \int_{-1}^1 J^{39}\prod_{t\in
  T\setminus\{a_{12}\}}w(t)
\,\,da_{12} = O\Bigl(J^{39}\prod_{t\in T\setminus\{a_{12}\}}w(t)\Bigr).\]
The second term of the integrand in (\ref{bigint}) can thus be
absorbed into the third term.
 

Finally, one easily computes the
integral of the third term in (\ref{bigint}) to be
$O(J^{39}X^{39/40})$.  We thus obtain, for any $v\in V^{(i)}$, that
\begin{equation}\label{obtain}
N^\ast(V^{(i)};X) = 
\frac1{n_i}\cdot\Vol(\RR_X(v)) 
+ O(J^{39}X^{39/40}/M_i(J)).
\end{equation}
\end{proof}



Note that the above proposition counts all integer points in $\RR_X(v)$
satisfying $a_{12}\neq 0$, not just the irreducible ones.
However, in this regard we have the following lemma:



\begin{lemma}\label{3reducible}
Let $v\in H\cap V^{(i)}$.  Then the number of 
$(A,B,C,D)\in\FF v$ such that $a_{12}\neq 0$, $|\Disc(A,B,C,D)|<X$, and 
$(A,B,C,D)$ is not irreducible is $o(X)$.
\end{lemma}

Lemma~\ref{3reducible} will in fact follow from a stronger lemma.   We
say that an element $(A,B,C,D)\in V_\Z$ is {\it absolutely irreducible} if
it is irreducible and the fraction field of its associated quintic ring is 
an $S_5$-quintic field (equivalently, if
the fields of definition of its common zeroes in $\P^3$ are 
$S_5$-quintic fields).
Then we have the following lemma, whose proof is postponed to Section~3:
\begin{lemma}\label{3reducible2}
Let $v\in H\cap V^{(i)}$.  Then the number of 
$(A,B,C,D)\in\FF v$ such that $a_{12}\neq 0$, $|\Disc(A,B,C,D)|<X$, and 
$(A,B,C,D)$ is not absolutely irreducible is $o(X)$.
\end{lemma}

Therefore, to prove Theorem~\ref{cna}, it remains only to compute the
fundamental volume $\Vol(\RR_X(v))$ for $v\in V^{(i)}$.  This is
handled in the next subsection.



\subsection{Computation of the fundamental volume}\label{volcomp}

In this subsection, we compute $\Vol(\RR_X(v))$, where $\RR_X(v)$ is
defined as in Section~\ref{mainterm}.  We will see that this volume
depends only on whether $v$ lies in $V^{(0)}$, $V^{(1)}$, or
$V^{(2)}$; here $V^{(i)}$ again denotes the $G_\R$-orbit in $V_\R$
consisting of those elements $(A,B,C,D)$ having nonzero discriminant
and possessing $5-2i$ real zeros in $\P^3$.

Before performing this computation, we first state two 
propositions regarding
the group $G=\GL_4\times \SL_5$ and its 40-dimensional representation $V$.



\begin{proposition}\label{covering}
The group $G_\R$ acts transitively on $V^{(i)}$, and the isotropy groups
for $v\in V^{(i)}$ are given as follows: 

\vspace{.1in}
$(${\em i}$)$ \,\,\,\,$S_5$, if $v\in V^{(0)}$; 

\vspace{.1in}
$(${\em ii}$)$ \,\,$S_3\times C_2$, if $v\in V^{(1)}$; and 

\vspace{.1in}
$(${\em iii}$)$ $D_4$, if $v\in V^{(2)}$.
\end{proposition} \vspace{.02in}
In view of Proposition~\ref{covering}, it will be convenient to use
the notation $n_i$ to denote the order of the stabilizer of any vector
$v\in V^{(i)}$.  Proposition~\ref{covering} implies that we have
$n_0=120$, $n_1=12$, and $n_2=8$.

Now define the usual subgroups $N$, $\bar N$, $A$, and $\Lambda$ of
$G_\R$ as follows:
\begin{eqnarray*}\label{subgroups}
N\,&=&\{n(x_1,x_2,\ldots,x_{16}):x_i\in\R \},\;\mbox{where}\,\\ 
{}&{}&
n({\boldmath{x}})={\footnotesize
\left(\left(\begin{array}{cccc} 1         & {x_1} & {x_2} & {x_3} \\ 
{} & 1 
& {x_4} &{x_5} \\ {}&{}&1&{x_6}\\{}&{}&{}&1 \end{array}\right),
\left(\begin{array}{ccccc} 1 & {x_7} & {x_8} &{x_9} &{x_{10}}\\ 
{}& 1 &{x_{11}} &{x_{12}}&{x_{13}}\\{}&{}&
1&{x_{14}}&{x_{15}}\\{}&{}&{}&1&{x_{16}}\\ 
{}&{}&{}&{}&1\end{array} \right)\right)};\,\, \\
\bar N\,&=&\{\bar n(u_1,u_2,\ldots,u_{16}):u_i\in\R \},\;\mbox{where}\,\\ 
{}&{}&
\bar n({\boldmath{u}})={\footnotesize
\left(\left(\begin{array}{cccc} 1         & {} & {} & {} \\ {u_1} & 1 
& {} &{} \\ {u_2}&{u_3}&1&{}\\{u_4}&{u_5}&{u_6}&1 \end{array}\right),
\left(\begin{array}{ccccc} 1 & {} & {} &{} &{}\\ 
{u_7}& 1 &{} &{}&{}\\{u_8}&{u_9}&
1&{}&{}\\{u_{10}}&{u_{11}}&{u_{12}}&1&{}\\ 
{u_{13}}&{u_{14}}&{u_{15}}&{u_{16}}&1\end{array} \right)\right)};\,\, \\
A&=&\{a(t_1,t_2,\ldots,t_7):t_1,t_2,\ldots,t_7\in\R_+\},\;\mbox{where}\,\\ 
{}&{}&a(\lambda,{\boldmath{t}})={\footnotesize
\left(\left(\begin{array}{cccc} t_1 
 & 
{} & {} & {} \\[.04in] {} & t_2/t_1
& {} &{} \\[.04in] {}&{}& t_3/t_2 
&{}\\[.04in] {}&{}&{}&
1/t_3  \end{array}\right),
\left(\begin{array}{ccccc} t_4 
& {} & {} &{} &{}\\[.04in] 
{}& t_5/t_4&{} &{}&{}\\[.04in]{}&{}&
t_6/t_5&{}&
{}\\[.04in]
{}&{}&{}& t_7/t_6 
&{}\\[.04in] {}&{}&{}&{}& 1/t_7 \end{array} \right)\right)};\,\,\\
\Lambda\,&=&\{\{\lambda:\lambda>0\},\;\mbox{where}\,\\
{}&{}& \lambda \mbox{ acts by }{\footnotesize
\left(\left(\begin{array}{cccc} \lambda & {} & {} & {} \\ {} & \lambda 
& {} &{} \\ {}&{}&\lambda&{}\\{}&{}&{}&\lambda \end{array}\right),
\left(\begin{array}{ccccc} 1 & {} & {} &{} &{}\\ 
{}& 1 &{} &{}&{}\\{}&{}&
1&{}&{}\\{}&{}&{}&1&{}\\ 
{}&{}&{}&{}&1\end{array} \right)\right)}.
\end{eqnarray*}

We define an invariant measure $dg$ on $G_\R$ by
\begin{equation}
\int_G f(g) dg =
\int_{\R_+^\times}\int_{\R_+^{\times 7}}
\int_{\R^4}\int_{\R^4}
f(n(x)\bar n(u)a(t)\lambda) 
                               \,dx\, du\, d^\times t\, d^\times\lambda.
\end{equation}
With this choice of Haar measure on $G_\R$, it is known that
$$\int_{G_\Z \backslash G^{\pm1}_\R} dg =
[\zeta(2)\zeta(3)\zeta(4)]\cdot[\zeta(2)\zeta(3)\zeta(4)\zeta(5)],$$
where $G^{\pm1}_\R\subset G_\R$ denotes the subgroup $\{(g_4,g_5)\in
G_\R:\det(g_4)=\pm1\}$ (see, e.g., \cite{Langlands}). 

Now let $dy=dy_1\, dy_2 \cdots dy_{40}$ be the standard Euclidean measure on 
$V_\R$.  Then we have:

\begin{proposition}\label{volumes2}
For $i=0${\rm ,} $1${\rm ,} or $2${\rm ,} 
let $f\in C_0(V^{(i)})${\rm ,} and let $y$ denote any
element of $V^{(i)}$.  Then
\begin{equation}\label{twenty}
\int_{g\in G_\R} f(g\cdot y)dg
\,\,=\,\, \frac{\,n_i}{20}\cdot
\int_{v\in V^{(i)}}|\Disc(v)|^{-1}f(v)\,dv . 
\end{equation}
\end{proposition}

\pagebreak
\begin{proof}
Put \[(z_1,\ldots,z_{40})=n(x) \bar n(u) a(t)\cdot y.\]  
Then the form $\Disc(z)^{-1} dz_1\wedge\cdots\wedge dz_{12}$ is a $G_\R$-invariant
measure, and so we must have
\[\Disc(z)^{-1} dz_1\wedge\cdots\wedge dz_{40}=c\,\,dx\wedge du\wedge d^\times t\wedge 
d^\times\lambda\] 
for some constant factor $c$.  An explicit Jacobian calculation shows
that $c=-20$.  (To make easier the calculation, we note that it
suffices to check this on any fixed
representative $y$ in $V^{(0)}$, $V^{(1)}$, or $V^{(2)}$.) 
By Proposition~\ref{covering}, the group
$G_\R$ is an $n_i$-fold covering of $V^{(i)}$ via the map $g\rightarrow
g\cdot y$.  Hence 
\[\int_{G_\R} f(g\cdot y)dg = \frac{n_i}{20}\cdot\int_{V^{(i)}}|\Disc(v)|^{-1}f(v)dv.\]
as desired.
\end{proof}

Finally, for any vector $y\in V^{(i)}$ of absolute discriminant 1,
we obtain using Proposition~\ref{volumes2} that
\[\frac1{n_i}\cdot\Vol(\RR_X(y)) = \frac{20}{n_i} \int_{1}^{X^{1/40}}
\lambda^{40}d^\times \lambda
               \int_{G_\Z\backslash G^{\pm1}_\R}dg = 
\frac{\zeta(2)^2\zeta(3)^2\zeta(4)^2\zeta(5)}{2 n_i}X,\]
proving Theorem~\ref{cna}.




\subsection{Congruence conditions}

We may prove a version of Theorem~\ref{cna} for a set in
$V^{(i)}$ defined by a finite number of congruence conditions:

\begin{theorem}\label{cong}
Suppose $S$ is a subset of $V^{(i)}_\Z$ defined by finitely many
congruence conditions. Then we have 
\begin{equation}\label{ramanujan}
\lim_{X\rightarrow\infty}\frac{N(S\cap V^{(i)};X)}{X} 
  = \frac{\zeta(2)^2\zeta(3)^2\zeta(4)^2\zeta(5)}{2n_i}
  \prod_{p} \mu_p(S),
\end{equation}
where $\mu_p(S)$ denotes the $p$-adic density of $S$ in $V_\Z$,
and $n_i=120$, $12$, or $8$ for
$i=0$, $1$, or $2$ respectively.  
\end{theorem}

To obtain Theorem~\ref{cong}, suppose $S$ is defined by congruence
conditions modulo some integer $m$.  Then $S$ may be viewed as the
union of (say) $k$ translates $L_1,\ldots,L_k$ of the lattice $m\cdot
V_\Z$.  For each such lattice translate $L_j$, we may use formula
(\ref{avg}) and the discussion following that formula to compute
$N(S;X)$, but where each $d$-dimensional volume is scaled by a factor
of $1/m^d$ to reflect the fact that our new lattice has been scaled by
a factor of $m$.  For a fixed value of $m$, we thus obtain
\begin{equation}\label{sestimate}
N(L_j;X) = m^{-40}\,\Vol(\RR_X(v)) + O(m^{-39}J^{39}X^{39/40}/M_i(J))
\end{equation}
for $v\in V^{(i)}$, where the implied constant is also 
independent of 
$m$ provided $m=O(X^{1/40})$.  Summing (\ref{sestimate}) over $j$, and
noting that $km^{-40}=\prod_p\mu_p(S)$, yields (\ref{ramanujan}).

\section{Quadruples of $5\times5$ skew-symmetric matrices 
and Theorems 1--4}

Theorems~\ref{main} and~\ref{cna} of the previous section
now immediately imply the following.

\begin{theorem}\label{ringwithres}
Let $M_{5}^{*(i)}(\xi,\eta)$ denote the number of isomorphism classes
of pairs $(R,R')$ such that
$R$ is an order in an $S_5$-quintic 
field with $5-2i$ real embeddings, $R'$ is a sextic resolvent ring of
$R$, and $\xi<\Disc(R)<\eta$.  Then
\[
\begin{array}{rlcl}\label{dodqrr}
\rm{(a)}& \displaystyle{\lim_{X\rightarrow\infty} \frac{M_5^{*(0)}(0,X)}{X}}
   &=&\! \displaystyle{\frac{\zeta(2)^2\zeta(3)^2\zeta(4)^2\zeta(5)}{240}}; \\[.1in]
\rm{(b)}& \displaystyle{\lim_{X\rightarrow\infty}
   \frac{M_5^{*(1)}(-X,0)}{X}} &=&\!
   \displaystyle{\frac{\zeta(2)^2\zeta(3)^2\zeta(4)^2\zeta(5)}{24}}; \\[.1in]
\rm{(c)}&\displaystyle{\lim_{X\rightarrow\infty} \frac{M_5^{*(2)}(0,X)}{X}} 
   &=& \!\displaystyle{\frac{\zeta(2)^2\zeta(3)^2\zeta(4)^2\zeta(5)}{16}}.
\end{array}
\]
\end{theorem}


To obtain finer asymptotic information on the distribution of 
quintic rings (in particular, without the weighting by the number 
of sextic resolvents), we need to be able to count
irreducible equivalence classes in $V_\Z$ lying in certain subsets 
$S\subset V_\Z$.  If $S$ is defined, say, by {\it finitely many}
congruence conditions, then Theorem~\ref{cong} applies in that case. 

However, the set $S$ of elements $(A,B,C,D)\in V_\Z$ corresponding to
maximal quintic orders is defined by infinitely many congruence conditions
(see \cite[\S12]{Bhargava4}).  To prove that (\ref{ramanujan}) still holds
for such a set, we require a uniform estimate on the error term when
only finitely many factors are taken in (\ref{ramanujan}).  This
estimate is provided in Section~3.1.  In Section~3.2, we prove
Lemma~\ref{3reducible2}.  Finally, in Section~3.3, we complete the
proofs of Theorems~1--4.

\subsection{A uniformity estimate}

As in~\cite{Bhargava4}, for a prime number $p$ let us denote by 
$\mathcal U_p$ the set of all $(A,B,C,D)\in V_\Z$ corresponding to 
quintic orders $R$ that are maximal at $p$.  
Let $\mathcal W_p=V_\Z-\mathcal U_p$.  
In order to apply a sieve to obtain Theorems 1--4, we
require the following proposition, analogous to Proposition~1 in \cite{DH}
and Proposition~23 in~\cite{Bhargava5}.

\begin{proposition}\label{errorestimate}
$N(\mathcal W_p;X) = O(X/p^2)$, 
where the implied constant is independent of $p$.
\end{proposition}

\begin{proof} 
We begin with the following lemma.

\begin{lemma}\label{upestimate}
  The number of maximal orders in quintic fields, up to isomorphism, 
having absolute discriminant less than $X$ is $O(X)$.
\end{lemma}
Lemma~\ref{upestimate} 
follows immediately from Theorem~\ref{ringwithres}, since we have shown that
every quintic ring has a sextic resolvent ring
(\cite[Corollary 4]{Bhargava4}).

To estimate $N(\mathcal W_p;X)$ using Lemma~\ref{upestimate}, we only need to
know that (a) the number of subrings of index $p^k$ ($k\geq 1$) in a
maximal quintic ring $R$ does not grow too rapidly with $k$; and
(b) the 
number of sextic resolvents that such a subring possesses is also
not too large relative to $p^{k}$.  For (a), an even stronger result
than we need here has recently been proven 
in the Ph.D. thesis~\cite{Jos} of Jos Brakenhoff,
who shows that the number of orders having index $p^k$ in a
maximal quintic ring $R$ is at most
$O(p^{\min\{2k-2,\frac{20}{11}k\}})$ for $k\geq1$, where the implied
constant is independent of $p$, $k$, and $R$.  Any such order
will of course have discriminant $p^{2k}\:\!\Disc(R)$.  As for (b), it follows
from \cite[Proof of Corollary 4]{Bhargava4} that the number of sextic
resolvents of a quintic ring having content $n$ is $O(n^6)$; moreover,
the number of sextic resolvents of a maximal quintic ring is 1.
(Recall that the {\it content} of a quintic ring $R$ is the largest
integer $n$ such that $R=\Z+nR'$ for some quintic ring $R'$.)

Since every content~$n$ quintic ring $R$ arises as $\Z+nR'$ for a
unique content~1 quintic ring $R'$, and $\Disc(R)=n^8\;\!\Disc(R')$,
we have
\[N(\mathcal W_p;X) 
  =\sum_{n=1}^\infty \frac{O(n^6)}{n^8}\sum_{k=1}^\infty 
       \frac{O(p^{\min\{2k-2,\frac{20}{11}k\}})}
{p^{2k}}O(X)
   = O(X/p^2),\]
as desired.  
\end{proof}

\subsection{Proof of Lemma~\ref{3reducible2}}\label{lemmaproof}

We say a quintic ring is an {\it $S_5$-quintic ring} if it is an
order in an $S_5$-quintic field.  To prove Lemma~\ref{3reducible2},  
we wish
to show that the expected number of
integral elements $(A,B,C,D)\in\FF v$ ($v\in V^{(i)}$) 
that correspond to quintic rings
that are not $S_5$-quintic rings, and such that $|\Disc(A,B,C,D)|< X$
and $a_{12}\neq 0$, is $o(X)$.

Now if a quintic ring $R=R(A,B,C,D)$ is not an $S_5$-quintic ring, then we
claim that either the splitting type $(1112)$ or $(5)$ does not occur in
$R$.  Indeed, if both of these splitting types occur in $R$, then $R$
is clearly a domain (since $R/pR\cong \F_{p^5}$ for some prime $p$)
and the Galois group associated with the quotient field of $R$ then 
must contain a 5-cycle and a transposition, implying that the Galois group is
in fact $S_5$.

Therefore, to obtain an upper bound on the expected number of
integral elements $(A,B,C,D)\in\FF v$ such that $R(A,B,C,D)$ is not an
$S_5$-quintic ring, $|\Disc(A,B,C,D)|< X$, and $a_{12}\neq 0$, we may
simply count those quintic rings in which $p$ does not split as
$(1112)$ in $R$ for any prime $p<N$ 
and those quintic rings for which $p$ does not have splitting type
$(5)$ for any prime $p<N$ (for some sufficiently large $N$).  
Now the $p$-adic density $\mu_p(T_p(1112))$
in $V_\Z$ of the
set of those $(A,B,C,D)\in T_p(1112)$ approaches $1/12$ as
$p\to\infty$ while the
$p$-adic density $\mu_p(T_p(5))$ of those $(A,B,C,D)\in T_p(5)$ 
approaches $1/5$ as $p\to\infty$ 
(by~\cite[Lemma~20]{Bhargava4}).  We conclude from (\ref{ramanujan})
that the total number of such $(A,B,C,D)\in \FF v$ 
that do not lie in
$T_p(1112)$ for any $p<N$ or do not lie in $T_p(5)$ for any $p<N$, 
and satisfy
$|\Disc(A,B,C,D)|<X$ for sufficiently large $X=X(N)$, is at most
\[ \frac{\zeta(2)^2\zeta(3)^2\zeta(4)^2\zeta(5)}{2n_i}
  \Bigl(\prod_{p<N} \bigl(1-\mu_p(T_p(1112))\bigr)
+ \prod_{p<N} \bigl(1-\mu_p(T_p(5))\bigr)\Bigr)X+o(X).\]
Letting $N\to\infty$, we see that asymptotically the above count of
$(A,B,C,D)$ is less than $cX$ for any fixed positive constant $c$, and
this completes the proof.

\subsection[\vspace{.025in}Proofs of Theorems 1--4]{Proofs of Theorems 1--4}

\noindent {\bf Proof of Theorem 1:}
Again, let $\mathcal U_p$ denote the set of all
$(A,B,C,D)\in V_\Z$ that correspond to pairs $(R,R')$ where $R$ is maximal
at $p$, and let $\mathcal U=\cap_p \mathcal U_p$.  Then $\mathcal U$
is the set of $(A,B,C,D)\in V_\Z$ corresponding to maximal quintic rings $R$.
In~\cite[Theorem~21]{Bhargava4}, we determined the $p$-adic density 
$\mu(\mathcal U_p)$ of $\mathcal U_p$:
\begin{equation}\label{totaludensity}
\mu(\mathcal U_p)=(p-1)^8p^{12}(p+1)^4(p^2+1)^2(p^2+p+1)^2
(p^4+p^3+p^2+p+1)(p^4+p^3+2p^2+2p+1)\,/\,p^{40}\,.
\end{equation}
Suppose $Y$ is any positive integer.  It follows from
(\ref{ramanujan}) and (\ref{totaludensity}) that 
\[\begin{array}{rcl}
& &\!\!\!\!\!\!\!\!\!\!\!\!\!\!\!\!\!\!\!
\displaystyle{\lim_{X\rightarrow\infty} \frac{N(\cap_{p<Y} \mathcal U_p\cap
  V^{(i)};X)}{X}} \\[.175in]
 & = & \displaystyle{\frac{\zeta(2)^2\zeta(3)^2\zeta(4)^2\zeta(5)}{2n_i}
    \prod_{p<Y}[p^{-28}\:(p^2-1)^2(p^3-1)^2(p^4-1)^2(p^5-1)(p^5+p^3-p-1)].}
\end{array}
\]
Letting $Y$ tend to $\infty$, we obtain immediately that
\[
\begin{array}{rcl}
& & \!\!\!\!\!\!\!\!\!\!\!\!\!\!\!\!\!\!\!
\displaystyle{\limsup_{X\rightarrow\infty}
\frac{N(\mathcal U\cap V^{(i)};X)}{X}}   \\[.175in]
 & \leq &\displaystyle{\frac{\zeta(2)^2\zeta(3)^2\zeta(4)^2\zeta(5)}{2n_i} 
    \prod_{p}[p^{-28}(p^2-1)^2(p^3-1)^2(p^4-1)^2(p^5-1)(p^5+p^3-p-1)]}
\\[.065in]
 & =& \displaystyle{\frac{\zeta(2)^2\zeta(3)^2\zeta(4)^2\zeta(5)}{2n_i} 
    \prod_p [(1-p^{-2})^2(1-p^{-3})^2(1-p^{-4})^2(1-p^{-5})
(1+p^{-2}-p^{-4}-p^{-5})]}.\\[.065in]
 & =& \displaystyle{\frac{1}{2n_i} \prod_p (1+p^{-2}-p^{-4}-p^{-5})}.
\end{array}
\]
To obtain a lower bound for $N(\mathcal U\cap V^{(i)};X)$, we note that
\[\bigcap_{p<Y} \mathcal U_p \subset 
(\mathcal U \cup \bigcup_{p\geq Y}\mathcal W_p).\]
Hence by Proposition~\ref{errorestimate}, 
\[\begin{array}{rcl}
& & \!\!\!\!\!\!\!\!\!\!\!\!\!\!\!\!\displaystyle{\lim_{X\rightarrow\infty}
\frac{N(\mathcal U\cap V^{(i)};X)}{X}} \\[.175in]
&\!\!\!\!\!\!\!\!\!\!\geq& 
\!\!\!\!\!\!\displaystyle{\frac{\zeta(2)^2\zeta(3)^2\zeta(4)^2\zeta(5)}
{2n_i} 
   \prod_{p<Y}[p^{-28}(p^2-1)^2(p^3-1)^2(p^4-1)^2(p^5-1)(p^5+p^3-p-1)] 
      - O(\sum_{p\geq Y} p^{-2}).}
\end{array}
\]
Letting $Y$ tend to infinity completes the proof of Theorem 1.
{$\Box$ \vspace{2 ex}}

%

\noindent {\bf Proof of Theorem 2:} 
For each (isomorphism class of) quintic ring $R$, we make a choice of
sextic resolvent ring $R'$, and let $S\subset V_\Z$ denote the set of
all elements in $V_\Z$ that yield the pair $(R,R')$ (under the
bijection of Theorem~\ref{main}) for some $R$.  Then we
wish to determine $N(S\cap V^{(i)};X)$ for $i=0,1,2$; by equation
(\ref{ramanujan}), this amounts to determining the $p$-adic density
$\mu_p(S)$ of $S$ for each prime $p$ for our choice of $S$.  In this
regard we have the following formula, which follows easily from the
arguments in \cite[Proof of Lemma~20]{Bhargava4}:
\begin{equation}\label{orderdensity}
\mu_p(S)\,=\, 
\frac{|G(\F_p)|}{\Disc_p(R)\cdot|\Aut_{\Z_p}(R)|}.
\end{equation}
Combining (\ref{ramanujan}) and (\ref{orderdensity}) together with the
fact that 
$$|G(\F_p)|=(p-1)^8\; p^{16}\;(p+1)^4\;(p^2+1)^2\;(p^2+p+1)^2\;
(p^4+p^3+p^2+p+1),$$
and proceeding as in Theorem~1,
now yields Theorem~2.
{$\Box$ \vspace{2 ex}}

\noindent {\bf Proof of Theorem 3:} Let $K_5$ be an $S_5$-quintic
field, and $K_{120}$ its Galois closure.  It is known that the Artin symbol
$(K_{120}/p)$ equals $\langle e \rangle$, $\langle (12) \rangle$,
$\langle (123) \rangle$, $\langle (1234) \rangle$, $\langle (12345)
\rangle$, $\langle (12)(34) \rangle$, or $\langle (12)(345)\rangle$
precisely when the splitting type of $p$ in $R$ is $(11111)$,
$(1112)$, $(113)$, $(14)$, $(5)$, $(122)$, or $(23)$ respectively,
where $R$ denotes the ring of integers in $K_5$.  As
in~\cite{Bhargava4}, let $U_p(\sigma)$ denote the set of all $(A,B,C,D)\in
V_\Z$ that correspond to maximal quintic rings $R$ having a specified
splitting type $\sigma$ at $p$.
Then by the same argument as in the proof of
Theorem~1, we have
\[\lim_{X\rightarrow\infty}\frac{N(U_p(\sigma)\cap V^{(i)};X)}{X}
  = \frac{\zeta(2)^2\zeta(3)^2\zeta(4)^2\zeta(5)}{2n_i}
\mu_p(U_p(\sigma))\prod_{q\neq p} \mu_q(\mathcal U_q) .
\] 
On the other hand, Lemma 20 of \cite{Bhargava4} gives the $p$-adic
densities of $U_p(\sigma)$ for all splitting and ramification types
$\sigma$; in particular, the values of $\mu_p(U_p(\sigma))$ for
$\sigma = (11111)$, $(1112)$, $(113)$, $(14)$, $(5)$, $(122)$, or
$(23)$ are seen to occur
in the ratio $1\!:\!10\!:\!20\!:\!30\!:\!24\!:\!15\!:\!20$ for any
value of $p$; this is the desired result.  {$\Box$ \vspace{2 ex}}

\noindent {\bf Proof of Theorem 4:} This follows immediately from
Theorem~1, Lemma~\ref{hard}, and Lemma~\ref{3reducible2}.


\section*{Acknowledgments}
 

I am very grateful to B.\ Gross, H.\ W.\ Lenstra, P.\ Sarnak, A.\
Shankar, A.\ Wiles, and M.\ Wood for many helpful discussions during
this work.
I am also very thankful to the Packard Foundation for their kind
support of this project.


\end{document}